\documentclass[twoside,a4paper,12pt,centertags]{amsart}
\usepackage{amsmath,amssymb,verbatim,vmargin}
\usepackage[bookmarks=true]{hyperref}   

\newcommand{\norm}{\,|\!|\,}

\newcommand{\loc}{\text{\rm loc}}
\newcommand{\osc}{\text{\rm osc}}

\numberwithin{equation}{section} 
\newtheorem{theorem}{\textbf{Theorem}}[section]
\newtheorem{lemma}{\textsc{Lemma}}[section]

\newtheorem{corollary}{\textsc{Corollary}}[section]
\theoremstyle{definition}

{\theoremstyle{remark} \newtheorem{remark}{Remark}[section]}

\newcommand{\divo}{\textnormal{div}_{H}}

\def\eqn#1$$#2$${\begin{equation}\label#1#2\end{equation}}

\newcommand{\R}{\mathbb R}

\newcommand{\Om}{\Omega}

\newcommand{\X}{\mathfrak X}
\newcommand{\Xu}{\mathfrak X u}

\newcommand{\dx}{\, dx}

\renewcommand{\epsilon}{\varepsilon}

\def\mvint_#1{\mathchoice
          {\mathop{\vrule width 6pt height 3 pt depth -2.5pt
                  \kern -8pt \intop}\nolimits_{\hspace{-1ex}#1}}%
          {\mathop{\vrule width 5pt height 3 pt depth -2.6pt
                  \kern -6pt \intop}\nolimits_{\hspace{-1ex}#1}}%
          {\mathop{\vrule width 5pt height 3 pt depth -2.6pt
                  \kern -6pt \intop}\nolimits_{\hspace{-1ex}#1}}%
          {\mathop{\vrule width 5pt height 3 pt depth -2.6pt
                  \kern -6pt \intop}\nolimits_{\hspace{-1ex}#1}}}

\newcommand{\deltaX}{\big(\delta +|\Xu|^2\big)}

\newcommand{\weight}{\big(\delta+|\Xu|^2\big)^\frac{p-2}{2}}

 \newcommand{\intav}{-\hskip -1.1em\int} \newbox\tratto
\setbox\tratto=\hbox{{\count0=0\dimen0=-,9pt\dimen1=1,1pt
\loop\ifnum\count0<11 \advance \count0 by1 \vrule width.51pt
height\dimen1 depth\dimen0\kern-0.17pt \advance\dimen0
by-0.05pt\advance\dimen1 by0.1pt\repeat
\loop\ifnum\count0<21\advance \count0 by1 \vrule width.6pt
height\dimen1 depth\dimen0\kern-0.2pt \advance\dimen0
by-0.1pt\advance\dimen1by0.05pt\repeat}}

\newcommand{\medint}{\displaystyle\copy\tratto\kern-10.4pt\int\limits}

\begin{document}

\title[Variational problems in the Heisenberg group]
{$C^{1,\alpha}$-Regularity for variational problems in the Heisenberg group}
\author{Shirsho Mukherjee and Xiao Zhong}
\thanks{Mukherjee was supported by the European Union’s Seventh Framework Programme, Marie Curie Actions-Initial Training Network, under Grant Agreement No. 607643, “Metric
Analysis For Emergent Technologies (MAnET)”. Zhong was supported by the Academy of Finland, the Centre of Excellence in Analysis and Dynamics Research and the Academy Project (project 308759).}
\date{\today}
\address[S.\ Mukherjee]{Department of Mathematics and Statistics,
P.O.Box 35 (MaD), FIN-40014 University of Jyv\"askyl\"a, Finland}
\email{shirsho.s.mukherjee@jyu.fi}
\address[X.\ Zhong]{Department of Mathematics and Statistics, University of Helsinki, P.O. Box 68 (Gustaf H\"allstr\"omin katu 2b)
FIN-00014 University of Helsinki, Finland}
\email{xiao.x.zhong@helsinki.fi}


\begin{abstract}
We study the regularity of minima of scalar variational
integrals of $p$-growth, $1<p<\infty$, in the Heisenberg group
and prove the H\"older continuity of horizontal gradient of minima.
\end{abstract}

\maketitle


\section{Introduction\label{section:intro}}

Following \cite{Zh}, 
we continue to study in this paper the regularity of minima of scalar variational
integrals in the Heisenberg group ${\mathbb H}^n, n\ge 1$. 
Let $\Omega$ be a
domain in ${\mathbb H}^n$ and $u:\Omega\to {\mathbb R}$ a function. 
We denote by $\X u=(X_1u, X_2u,\ldots,X_{2n}u)$ the horizontal gradient
of $u$.
We study  the following variational problem
\begin{equation}\label{functional}
I(u)=\int_\Omega f(\X u)\, dx,  
\end{equation} 
where the convex integrand function $f\in C^2({\mathbb R}^{2n};{\mathbb
R})$ is of $p$-growth, $1<p<\infty$. It satisfies the following
growth and ellipticity conditions
\begin{equation}\label{structure}
\begin{aligned}(\delta+| z|^2)^{\frac{p-2}{2}} |\xi|^2\le
\langle D^2f(z) &\xi,\xi\rangle \le L(\delta+|
z|^2)^{\frac{p-2}{2}}| \xi|^2;\\
| Df(z)| & \le L(\delta+| z|^2)^{\frac{p-2}{2}}|z|
\end{aligned}
\end{equation}
for all $z, \xi\in \R^{2n}$, where $\delta\ge 0, L\ge 1$ are
constants.

It is easy to prove that a function in the horizontal Sobolev space $HW^{1,p}(\Omega)$ is a local minimizer of functional (\ref{functional}) if and only if
it is a weak solution of 
the corresponding Euler-Lagrange equation of
(\ref{functional})
\begin{equation}\label{equation:main}
\divo \big(Df(\X u)\big)=\sum_{i=1}^{2n}X_i\big(D_i f(\X u)\big)=0.
\end{equation}
where $Df=(D_1 f, D_2f,\ldots,D_{2n}f)$ is the Euclidean gradient
of $f$. See Section \ref{section:preliminaries} for the definitions of horizontal Sobolev space $HW^{1,p}(\Omega)$, weak solutions and local minimizers. 

A prototype example of integrand functions satisfying conditions (\ref{structure})
is 
\[ f(z)=\big(\delta+| z|^2\big)^{\frac p 2}\]
for a constant $\delta\ge 0$. 
Then the Euler-Lagrange
equation (\ref{equation:main}) is reduced to the non-degenerate $p$-Laplacian equation
\begin{equation}\label{equation:p:nondeg}
\divo \big(\weight\X u\big)=0,\end{equation} when $\delta>0$,
and the $p$-Laplacian equation
\begin{equation}\label{equation:p}
\divo\big(|\X u|^{p-2}\X u\big)=0,
\end{equation}
when $\delta=0$. The weak solutions of equation (\ref{equation:p})
are called $p$-harmonic functions.

For the regularity of weak solutions of equation (\ref{equation:main}), the second author proved in \cite{Zh} the following theorem, Theorem 1.1 of \cite{Zh}, from which follows the Lipschitz continuity of weak solutions for 
all $1<p<\infty$. We remark that this result holds both for the non-degenerate case ($\delta>0$) and for the degenerate one ($\delta=0$). We also remark
that it holds under a bit more general growth condition on the integrand 
function $f$ than (\ref{structure}). Precisely, in \cite{Zh} the integrand function $f$ is assumed to satisfy
\begin{equation}\label{structureprime}
\begin{aligned}(\delta+| z|^2)^{\frac{p-2}{2}} |\xi|^2\le
\langle D^2f(z) &\xi,\xi\rangle \le L(\delta+|
z|^2)^{\frac{p-2}{2}}| \xi|^2;\\
| Df(z)| & \le L(\delta+| z|^2)^{\frac{p-1}{2}}
\end{aligned}
\end{equation}
for all $z, \xi\in \R^{2n}$, where $\delta\ge 0, L\ge 1$ are
constants.
\begin{theorem}\label{thm:lip}
Let $1<p<\infty$, $\delta\ge 0$ and $u\in HW^{1,p}(\Omega)$ be a weak solution of
equation (\ref{equation:main}) satisfying the structure condition (\ref{structureprime}). Then $\Xu \in
L^\infty_{\loc}(\Omega;{\mathbb R}^{2n})$. Moreover, for any
ball $B_{2r}\subset \Omega$, we have that
\begin{equation}\label{Xu:bdd}
\sup_{B_r}|\X u|\le c\Big(\intav_{B_{2r}} \deltaX^{\frac p 2}\, dx\Big)^{\frac 1 p},
\end{equation}
where $c>0$ depends only on $n,p,L$.
\end{theorem}

Here and in the following, the ball $B_r$ is defined with respect to the 
Carnot-Carath\`eodory metric (CC-metric) $d$; $B_{2r}$ is the double size ball with the same center, see Section \ref{section:preliminaries} for the
definitions.

The second author also proved in \cite{Zh} that the horizontal gradient of weak solutions of equation (\ref{equation:main}) is H\"older continuous when $p\ge
2$. We remark again that this result holds under the condition (\ref{structureprime}), and that it holds both for the non-degenerate case ($\delta>0$) and for the degenerate one ($\delta=0$).

\begin{theorem}\label{thm:holder}
Let  $2\le p<\infty$, $\delta\ge 0$ and $u\in HW^{1,p}(\Omega)$ be a weak
solution of equation (\ref{equation:main}) satisfying the structure condition 
(\ref{structureprime}). Then the horizontal
gradient $\X u$ is H\"older continuous. Moreover, there is a
positive exponent $\alpha=\alpha(n,p,L)\le 1$ such that for any
ball $B_{r_0}\subset \Omega$ and any $0<r\le r_0$, we have
\begin{equation}\label{Xu:holderprime} \max_{1\le l \le 2n} {\operatorname{osc}}_{B_r} X_l u
\le
c\Big(\frac{r}{r_0}\Big)^\alpha\Big(\intav_{B_{r_0}}\deltaX^{\frac{
p}{ 2}}\, dx\Big)^{\frac{1}{p}},\end{equation} where $c>0$ depends only on
$n,p, L$.
\end{theorem}

We refer to the paper \cite{Zh} and the references therein, e.g. \cite{H, F, F2, Kohn, BLU, Ca1, Ca2, CDG, CG, Fog, DM,
DM2, D, Ma, MZZ, MM} for the earlier work on the regularity of weak solutions of equation (\ref{equation:main}). 

The result in Theorem \ref{thm:holder} leaves open the H\"older continuity of horizontal gradient of weak solutions for equation (\ref{equation:main}) in the case $1<p<2$. In this paper, we prove that the same result holds for this case, under the structure condition (\ref{structure}). This is the main result of this paper.

\begin{theorem}\label{thm:main}
Let  $1< p<\infty$, $\delta\ge 0$ and $u\in HW^{1,p}(\Omega)$ be a weak
solution of equation (\ref{equation:main}) satisfying the structure condition (\ref{structure}). Then the horizontal
gradient $\X u$ is H\"older continuous. Moreover, there is a
positive exponent $\alpha=\alpha(n,p,L)\le 1$ such that for any
ball $B_{r_0}\subset \Omega$ and any $0<r\le r_0$, we have
\begin{equation}\label{Xu:holder} \max_{1\le l \le 2n} {\operatorname{osc}}_{B_r} X_l u
\le
c\Big(\frac{r}{r_0}\Big)^\alpha\Big(\intav_{B_{r_0}}\deltaX^{\frac{
p}{ 2}}\, dx\Big)^{\frac{1}{p}},\end{equation} where $c>0$ depends only on
$n,p, L$.
\end{theorem}

For $p\neq 2$, it is well known that weak solutions of
equations of type (\ref{equation:main}) in the Euclidean spaces
are of the class $C^{1,\alpha}$, that is, they
have H\"older continuous derivatives, see \cite{Ur, LU, E, Di, Le, To}.
The $C^{1,\alpha}$-regularity is optimal when $p>2$. This can been seen by examples. 
Theorem \ref{thm:main} shows that the regularity theory 
for equation (\ref{equation:main}) in the setting of Heisenberg group 
is similar to that in the setting of Euclidean spaces. 

The proof of Theorem \ref{thm:main}  is based on De Giorgi's method \cite{De}
and it works for all $1<p<\infty$. The approach is similar to that  of Tolksdorff \cite{To} and  Lieberman  \cite{Lieb}
in the setting of Euclidean spaces. The idea is to consider the double 
truncation of the horizontal derivative $X_lu, l=1,2,..., 2n$, of the weak solution $u$ to equation (\ref{equation:main}) satisfying the structure condition (\ref{structure}) with $\delta>0$
\[ v=\min\big( \mu(r)/8, \max(\mu(r)/4-X_lu,0)\big), \]
where 
\[ \mu(r)= \max_{1\le i\le 2n}\sup_{B_r}\vert X_i u\vert,\]
and $B_r\subset\Omega$ is a ball. 
The whole difficulties of this work lie in proving the following Caccioppoli type inequality for $v$. In the following lemma, $\eta\in C^\infty_0(B_r)$ is a 
non-negative cut-off function such that $0\le \eta\le 1$ in $B_r$, 
$\eta=1$ in $B_{r/2}$ and that $|\X \eta|\le 4/r$, $|\X \X \eta|\le 16n/r^2$, $| T\eta|\le 32n/r^2$ in $B_r$. 

\begin{lemma}\label{lemma:main}
Let $\gamma>1$ be a number. We have the following Caccioppoli type inequality
 \begin{equation*}
  \int_{B_r} \eta^{\beta+4}v^{\beta+2}|\X v|^2\, dx \leq c(\beta+2)^2
\frac{| B_r|^{1-1/\gamma}}{r^2}
  \mu(r)^4\Big(\int_{B_r}\eta^{\gamma\beta}v^{\gamma\beta}\, dx\Big)^{1/\gamma}
\end{equation*}
for all $\beta\ge 0$, where $c=c(n,p, L,\gamma)>0$.
\end{lemma}

The proof of Lemma \ref{lemma:main} is based on the integrability estimate for $Tu$, the vertical derivative of $u$, established in \cite{Zh}, see Lemma
\ref{cor:Tu:high:original} and Lemma \ref{cor:Tu:high}. To prove Lemma \ref{lemma:main}, we consider the equation for $X_lu$, see equations (\ref{equation:horizontal}) and (\ref{equation:horizontal2})
of Lemma \ref{ws:horizontal} in Section \ref{section:preliminaries}. We take the usual testing function 
\[\varphi=\eta^{\beta+4}v^{\beta+3}\] 
for equations
(\ref{equation:horizontal}) and (\ref{equation:horizontal2}), where $\beta\ge 0$. In the case $p\ge 2$,
when equation (\ref{equation:main}) is degenerate, the proof of Lemma \ref{lemma:main} is not difficult.  On the contrary, in the case $1<p<2$, when
equation (\ref{equation:main}) is singular, it is dedicated to prove the desired Caccioppoli inequality in Lemma \ref{lemma:main}. In order to prove Lemma \ref{lemma:main}, we prove two auxiliary lemmas, Lemma \ref{lem:start} and Lemma \ref{lem:tech}, where we establish the Caccioppoli type inequalities
for $\X u$ and $Tu$ involving $v$.  
The essential feature of these inequalities is that we add weights such as the powers
of $\vert\X u\vert$, in order to deal with the singularity of equation (\ref{equation:main}) in the case $1<p<2$. 
The proof of Lemma \ref{lemma:main} is given in Section \ref{proofoflemma}. 

Once Lemma \ref{lemma:main} is established, the proof of Theorem \ref{thm:main}
is similar to that in the setting of Euclidean spaces. We may follow the same line as that in \cite{Zh}. The proof of Theorem \ref{thm:main} is given in Section \ref{section:Holder}. The proof of the auxiliary lemma, Lemma \ref{lem:start}, is given in the Appendix. 

\section{Preliminaries}\label{section:preliminaries}
In this section, we fix our notation and introduce the Heisenberg
group ${\mathbb H}^n$ and the known results on the sub-elliptic equation (\ref{equation:main}).

Throughout this paper, $c$ is a positive constant, which 
may vary from line to line. Except explicitly being specified, it
depends only on the dimension $n$ of the Heisenberg group, and on the constants $p$ and $L$ in the structure
condition (\ref{structure}). But, it does not depend on $\delta$
in (\ref{structure}).

\subsection{Heisenberg group ${\mathbb H}^n$} The Heisenberg group ${\mathbb H}^n$ is identified with the Euclidean space ${\mathbb
R}^{2n+1}, n\ge 1$. The group multiplication is given by
\[ xy=(x_1+y_1, \dots, x_{2n}+y_{2n}, t+s+\frac{1}{2}
\sum_{i=1}^n (x_iy_{n+i}-x_{n+i}y_i))\] for points
$x=(x_1,\ldots,x_{2n},t), y=(y_1,\ldots,y_{2n},s)\in {\mathbb H}^
n$. The left invariant vector fields corresponding to the
canonical basis of the Lie algebra are
\[ X_i=\partial_{x_i}-\frac{x_{n+i}}{2}\partial_t, \quad
X_{n+i}=\partial_{x_{n+i}}+\frac{x_i}{2}\partial_t,\] and the only
non-trivial commutator
\[ T=\partial_t=[X_i,X_{n+i}]=X_iX_{n+i}-X_{n+i}X_i\]
for $1\le i\le n$. We denote by
$\X=(X_1, X_2,\ldots,X_{2n})$ the horizontal gradient.
The second horizontal derivatives are given by the horizontal Hessian
$\X \X u$ of a function $u$, with entries $X_i(X_j u), i,j=1,\ldots, 2n$. Note that it is not symmetric, in general.
The standard Euclidean gradient of a function $v$ in ${\mathbb R}^k$ is denoted by
$Dv=(D_1v,\ldots, D_kv)$ and the Hessian matrix by $D^2v$.

The Haar measure in ${\mathbb H}^n$ is the Lebesgue measure of
${\mathbb R}^{2n+1}$. We denote by $| E|$ the Lebesgue
measure of a measurable set $E\subset {\mathbb H}^n$ and by
\[ \intav_E f\, dx=\frac{1}{| E|}\int_E f\, dx\]
the average of an integrable function $f$ over set $E$.

A ball $B_\rho(x)=\{ y\in {\mathbb H}^n: d(y,x)<\rho\}$
is defined with respect to the Carnot-Carath\`eodory metric (CC-metric) $d$.
The CC-distance of two points in ${\mathbb H}^n$ is the length of the shortest
horizontal curve joining them.

Let $1\le p<\infty$ and $\Omega\subset {\mathbb H}^n$ be an open set. The horizontal Sobolev space $HW^{1,p}(\Omega)$ consists
of functions $u\in L^p(\Omega)$ such that the horizontal weak gradient $\X u$ is also in $L^p(\Omega)$.
$HW^{1,p}(\Omega)$, equipped with the norm
\[ \norm u\norm_{HW^{1,p}(\Omega)}=\norm u\norm_{L^p(\Omega)}+\norm \X u\norm_{L^p(\Omega)},\]
is a Banach space. 
$HW^{1,p}_0(\Omega)$ is the closure of $C^\infty_0(\Omega)$ in $HW^{1,p}(\Omega)$ with this norm.
We denote the local space by $HW^{1,p}_{\loc}(\Omega)$.

The following Sobolev inequality hold for functions $u \in HW^{1,q}_0(B_r)$, $1\le q< Q=2n+2$, 
\begin{equation}\label{sobolev}
\Big(\intav_{B_r}| u|^{\frac{Qq}{Q-q}}\, dx\Big)^{\frac{Q-q}{Qq}}\le c r\Big(\intav_{B_r}| \X u|^q\, dx\Big)^{\frac 1 q},
\end{equation}
where $B_r\subset {\mathbb H}^n$ is a ball and $c=c(n,q)>0$.

\subsection{Known results on sub-elliptic equation \eqref{equation:main}}
A function $u\in HW^{1,p}(\Omega)$ is a local minimizer of functional (\ref{functional}), that is,
\[ \int_\Omega f(\Xu)\, dx \le \int_\Omega f(\X u+\X \varphi)\, dx\]
for all $\varphi\in C_0^\infty(\Omega)$, if and only if it is a weak solution of equation (\ref{equation:main}), that is,
\[\int_\Omega \langle Df(\X u), \X\varphi\rangle\, dx=0\]
for all $\varphi\in C_0^\infty(\Omega)$.

In the rest of this subsection, $u\in HW^{1,p}(\Omega)$ is a weak
solution of equation (\ref{equation:main}) satisfying the structure condition (\ref{structure}) with $\delta>0$. By Theorem \ref{thm:lip}, we have that 
\[ \X u\in L^\infty_{\loc}(\Omega;{\mathbb R}^{2n}).\]
Thanks to this and to the fact that we assume $\delta>0$, equation (\ref{equation:main}) is uniformly elliptic. 
Then we can apply Capogna's results in \cite{Ca1}. Theorem 1.1 and Theorem 3.1 of \cite{Ca1} show that $\X u$ and $Tu$  are
H\"older continuous in $\Omega$, and that
\begin{equation}\label{apregularity} \X u\in
HW^{1,2}_{\loc}(\Omega;{\mathbb R}^{2n}), \quad Tu\in
HW^{1,2}_{\loc}(\Omega)\cap L^\infty_{\loc}(\Omega).\end{equation}
With the above regularity, we can easily prove the following three lemmas. They are 
Lemma 3.1, Lemma 3.2 and Lemma 3.3 of \cite{Zh}, respectively. We refer to
\cite{Zh} for the proofs.

\begin{lemma}\label{ws:horizontal}
Let $v_l=X_l u, l=1,2,\ldots,n$. Then $v_l$ is a weak solution of
\begin{equation}\label{equation:horizontal}
\sum_{i,j=1}^{2n}X_i\big(D_{j}D_if(\X u)
X_jv_l\big)+\sum_{i=1}^{2n}X_i\big(D_{n+l}D_i f(\X u)Tu\big)+T\big(D_{n+l}f(\X
u)\big)=0;
\end{equation}
Let $v_{n+l}=X_{n+l}u, l=1,2,\ldots,n$. Then $v_{n+l}$ is a weak
solution of
\begin{equation}\label{equation:horizontal2}
\sum_{i,j=1}^{2n}X_i\big(D_{j}D_if(\X u)
X_jv_{n+l}\big)-\sum_{i=1}^{2n}X_i\big(D_{l}D_i f(\X u)Tu\big)-T\big(D_{l}f(\X
u)\big)=0;
\end{equation}
\end{lemma}

\begin{lemma}\label{ws:T}
$Tu$ is a weak solution of
\begin{equation}\label{equation:T}
\sum_{i,j=1}^{2n} X_i\big(D_jD_i f(\Xu)X_j (Tu)\big)=0.
\end{equation}
\end{lemma}

\begin{lemma}\label{caccioppoli:T}
For any $\beta\ge 0$ and  all $\eta\in C^\infty_0(\Omega)$, we
have
\begin{equation*}
\int_\Omega \eta^2\weight | Tu|^\beta|\X (Tu)|^2\,
dx \le \frac{c}{(\beta+1)^2}\int_\Omega |\X\eta
|^2\weight| Tu|^{\beta+2}\, dx.
\end{equation*}
where $c=c(n,p,L)>0$.
\end{lemma}

The following lemma is Corollary 3.2 of \cite{Zh}. It shows the integrability 
of $Tu$. It is critical for the proof of the H\"older continuity of the horizontal gradient of solutions $u$ in \cite{Zh}.   
\begin{lemma}\label{cor:Tu:high:original}
For any $\beta\ge 2$ and all non-negative $\eta\in C^\infty_0(\Omega)$, we have that
\[
\int_\Omega\eta^{\beta+2}\weight| Tu|^{\beta+2}\, dx \le  c(\beta)K^{\frac{\beta+2}{2}}
\int_{spt(\eta)}\big(\delta+\vert\X u\vert^2\big)^{\frac{p+\beta}{2}}\, dx,
\]
where $K=\| \X\eta\|_{L^\infty}^2+\|\eta
T\eta\|_{L^\infty}$ and  $c(\beta)>0$ depends on $n,p,L$
and $\beta$.
\end{lemma}
In this paper, we need the following version of Lemma \ref{cor:Tu:high:original}, which is a bit stronger.
The reason that this stronger version holds is that we have a stronger structure condition (\ref{structure}) than that one (\ref{structureprime}) in \cite{Zh}.
\begin{lemma}\label{cor:Tu:high}
For any $\beta\ge 2$ and all non-negative $\eta\in C^\infty_0(\Omega)$, we have that
\[
\int_\Omega\eta^{\beta+2}\weight| Tu|^{\beta+2}\, dx \le c(\beta)K^{\frac{\beta+2}{2}}
\int_{spt(\eta)}\weight |\X u|^{\beta+2}\, dx,
\]
where $K=\| \X\eta\|_{L^\infty}^2+\|\eta
T\eta\|_{L^\infty}$ and  $c(\beta)>0$ depends on $n,p,L$
and $\beta$.
\end{lemma}

The following corollary follows easily from Lemma \ref{caccioppoli:T} and Lemma
\ref{cor:Tu:high}.
\begin{corollary}\label{cor:TX}
For any $q\ge 4$ and all non-negative $\eta\in C^\infty_0(\Omega)$, we have
\begin{equation*}
\begin{aligned}
\int_\Omega \eta^{q+2}\weight |Tu|^{q-2}|\X(Tu)|^2\dx
\leq c(q)K^\frac{q+2}{2}\int_{spt(\eta)}\weight |\X u|^q\dx,
\end{aligned}
\end{equation*}
where $K=\| \X\eta\|_{L^\infty}^2+\|\eta
T\eta\|_{L^\infty}$ and $c(q)=c(n,p,L,q)>0$.
\end{corollary}

In the rest of this subsection, we comment on the proof of Lemma \ref{cor:Tu:high}. 
The proof of Lemma \ref{cor:Tu:high} is almost the same as that of Lemma \ref{cor:Tu:high:original} in \cite{Zh}; it requires only minor modifications. Lemma \ref{cor:Tu:high:original} follows from two lemmas, that is, Lemma 3.4, Lemma 3.5 in \cite{Zh}. To prove Lemma \ref{cor:Tu:high}, we need stronger versions of Lemma 3.4 and Lemma 3.5 of \cite{Zh}, which we state here. The following lemma is a stronger version of Lemma 3.4 of \cite{Zh}.

\begin{lemma}\label{caccioppoli:horizontal:sigma}
For any $\beta\ge 0$ and all $\eta\in C^\infty_0(\Omega)$, we have
\begin{equation*}
\begin{aligned}
\int_{\Omega}\eta^2\weight |\X u|^\beta|\X\X u|^2\, dx\le&
 c\int_\Omega (| \X \eta|^2+\eta| T\eta|)\weight |\X u|^{\beta+2}\, dx\\
+c&(\beta+1)^4\int_\Omega\eta^2\weight |\X u|^\beta| Tu|^2\, dx,
\end{aligned}
\end{equation*}
where $c=c(n,p,L)>0$.
\end{lemma}

The proof of Lemma \ref{caccioppoli:horizontal:sigma} follows the same line as that of Lemma 3.4 of \cite{Zh} with minor modifications. To prove Lemma 3.4 of \cite{Zh}, one uses $\varphi=\eta^2\big(\delta+\vert\X u\vert^2\big)^{\beta/2}X_l u$ as a testing function for equations (\ref{equation:horizontal}) when
$l=1,2,...,n$ and for equation (\ref{equation:horizontal2}) when $l=n+1,n+2,...,2n$. Now, to prove Lemma \ref{caccioppoli:horizontal:sigma}, we use instead the testing function $\varphi=\eta^2\vert\X u\vert^\beta X_lu$. The proof then is the same as that of Lemma 3.4 of \cite{Zh} with obvious changes. To get through the proof, we remark that the structure condition (\ref{structure}) is essential. We omit the details of the proof of Lemma \ref{caccioppoli:horizontal:sigma}.  

The following lemma is a stronger version of Lemma 3.5 of \cite{Zh}. 
\begin{lemma}\label{caccioppoli:horizontal:T}
For any $\beta\ge 2$ and all non-negative $\eta\in C^\infty_0(\Omega)$, we have
\begin{equation*}
\begin{aligned}
\int_\Omega \eta^{\beta+2}&\weight| Tu|^{\beta}|\X\X u|^2\, dx\\
&\le
c(\beta+1)^2\|  \X\eta\| _{L^\infty}^2\int_\Omega\eta^\beta
\weight |\X u|^2  | Tu|^{\beta-2}| \X\Xu|^2\, dx,
\end{aligned}
\end{equation*}
where $c=c(n,p,L)>0$.
\end{lemma}

The proof of Lemma \ref{caccioppoli:horizontal:T} is almost the same as that of Lemma 3.5, with obvious minor changes. The only difference is that we use the structure condition (\ref{structure}) whenever the structure condition (\ref{structureprime}) is used in the proof of Lemma 3.5 in \cite{Zh}. We omit the details.

Once Lemma \ref{caccioppoli:horizontal:sigma} and Lemma \ref{caccioppoli:horizontal:T} are established, the proof of Lemma \ref{cor:Tu:high} is exactly the same as that of Lemma \ref{cor:Tu:high:original} in \cite{Zh}.

\section{Proof of the main lemma, Lemma \ref{lemma:main}}\label{proofoflemma}
Throughout this section, $u\in HW^{1,p}(\Omega)$ is a weak solution of equation (\ref{equation:main}) satisfying the structure condition (\ref{structure})
with $\delta>0$.
For any ball $B_r\subset \Omega$, we denote for $i=1,2,...,2n$,
\begin{equation}\label{def:mu}
\mu_i(r)=\sup_{{B_r}} \vert X_i u\vert, \quad \mu(r)=\max_{1\le i\le 2n}\mu_i(r).
\end{equation}
Now fix $l\in \{ 1,2,.., 2n\}$. We consider the following double truncation 
of $X_l u$ 
\begin{equation}\label{def:v}
v=\min\big(\mu(r)/8,\max(\mu(r)/4-X_lu,0)\big).
\end{equation}
We denote 
\begin{equation}\label{setE}
E=\{ x\in \Omega: \mu(r)/8< X_lu<\mu(r)/4\}.
\end{equation}
We note the following  trivial inequality, which we use several times in this section
\begin{equation}\label{comparable1}
\mu(r)/8\le |\Xu|\le (2n)^{1/2}\mu(r) \quad\text{in } E\cap B_r.
\end{equation}
It follows from the regularity results (\ref{apregularity}) that
\begin{equation}\label{reg:v} \X v\in L^2_\loc(\Omega;\R^{2n}), \quad Tv\in L^2_\loc(\Omega)
\end{equation}
and moreover
\begin{equation}\label{vis0}
\X v=\begin{cases}
-\X X_lu \ & \text{a.e. in } E;\\
0 &\text{a.e. in } \Omega\setminus E,
\end{cases}
\quad\quad 
T v=\begin{cases}
-T X_l u \ & \text{a.e. in } E;\\
0 &\text{a.e. in } \Omega\setminus E.
\end{cases}
\end{equation}

We note that the function 
\[ h(t)=\big(\delta+t^2\big)^{\frac{p-2}{2}} t^q\]
is non-decreasing on $[0,\infty)$ if $\delta\ge 0$ and $q\ge 0$ such that
$p+q-2\ge 0$. Thus we have the following inequality, which is used several times in this section 
\begin{equation}\label{trivial}
\big(\delta+\vert\X u\vert^2\big)^{\frac{p-2}{2}}\vert\X u\vert^q\le c(n,p,q)\big(\delta+\mu(r)^2\big)^{\frac{p-2}{2}} \mu(r)^q \quad \text{in }
B_r,
\end{equation}
where $c(n,p,q)=(2n)^{(q+p-2)/2}$ if $p\ge 2$ and $c(n,p,q)=(2n)^{q/2}$ if $1<p< 2$.

To prove Lemma \ref{lemma:main}, we need the following two lemmas. The first lemma is similar to Lemma 3.3 of \cite{Zh}. 
In this lemma, we prove a weighted Caccioppoli inequality for $\X u$ involving $v$. It has an
extra weight $\vert \X u\vert^2$, comparing to that in Lemma 3.3 of \cite {Zh}. 
This is essential for us to deal with the case $1<p<2$ when equation (\ref{equation:main}) is singular. 
The proof  is also similar to that of Lemma 3.3 of \cite{Zh}. It is standard, but lengthy. We give a detailed proof in the Appendix.
\begin{lemma}\label{lem:start}
Let $1<p<\infty$. For any $\beta\ge 0$ and all non-negative $\eta\in C^\infty_0(\Omega)$, we have that
\begin{equation}\label{ineq1}
\begin{aligned}
\int_\Omega \eta^{\beta+2}v^{\beta+2}&\weight |\X u|^2|\X\X u|^2\, dx\\
\le & \ c (\beta+2)^2\int_\Omega \eta^\beta \big(|\X \eta|^2+\eta| T\eta|\big)v^{\beta+2}
\weight |\X u|^4 \, dx\\
& +  c(\beta+2)^2\int_\Omega \eta^{\beta+2} v^ \beta\weight |\X u|^4| \X v|^2\, dx\\
& +
c\int_\Omega\eta^{\beta+2} v^{\beta+2}\weight |\X u|^2| Tu|^2\, dx,
\end{aligned}
\end{equation}
where $c=c(n,p,L)>0$.
\end{lemma}

In the following is the second lemma that we need for the proof of Lemma \ref{lemma:main}, where we prove  
a weighted Caccioppoli inequality for $Tu$ involving $v$. It has a weight $\vert \X u\vert^4$, which is needed for us to deal with the case $1<p<2$. 
To state the lemma, 
we fix, throughout the rest of this section, a ball $B_r\subset\Omega$ and a cut-off function $\eta\in C^\infty_0(B_r)$ that satisfies 
\begin{equation}\label{etaBr1}
0\le \eta\le 1 \quad \text{in }B_r, \quad
\eta=1 \quad \text{in }B_{r/2}
\end{equation} 
and  
\begin{equation}\label{etaBr2} 
|\X \eta|\le 4/r,\quad |\X \X \eta|\le 16n/r^2, \quad | T\eta|\le 32n/r^2\quad \text{in }B_r.
\end{equation}

\begin{lemma}\label{lem:tech}
Let $B_r \subset \Om$ be a ball and $\eta\in C^\infty_0(B_r)$ be a cut-off function satisfying \eqref{etaBr1} and \eqref{etaBr2}. 
Let $\tau \in (1/2,1)$ and $\gamma\in(1,2)$ be two fixed numbers. 
Then, for any $\beta \geq0$, 
we have  
\begin{equation}\label{eq:tech}
 \begin{aligned}
  \int_\Omega \eta^{\tau(\beta+2)+4}\,&v^{\tau(\beta+4)} 
\weight |\X u|^4|\X (Tu)|^2\, dx\\
 &\leq c(\beta+2)^{2\tau}\frac{|B_r|^{1-\tau}}{r^{2(2-\tau)}}\big(\delta+\mu(r)^2\big)^{\frac{p-2}{2}}
 \mu(r)^6
 \,J^{\tau},
 \end{aligned}
\end{equation}
where $c = c(n,p,L,\tau,\gamma)>0$ and 
\begin{equation}\label{def:J}
J= \int_{B_r} \eta^{\beta+4}v^{\beta+2} |\X v|^2\dx 
 \ +\ \mu(r)^4\frac{|B_r|^{1-\frac{1}{\gamma}}}{r^2}
\Big(\int_{B_r} \eta^{\gamma\beta}v^{\gamma\beta}\dx\Big)^\frac{1}{\gamma}.
\end{equation}
\end{lemma}
\begin{proof}
We denote by $M$ the left hand side of \eqref{eq:tech} 
\begin{equation}\label{def:K}
 M=\ \int_\Omega \eta^{\tau(\beta+2)+4}\,v^{\tau(\beta+4)} 
\weight |\X u|^4|\X (Tu)|^2\, dx,
\end{equation}
where $1/2<\tau<1$. 
We use the following function
$$\varphi = \eta^{\tau(\beta+2)+4}\,v^{\tau(\beta+4)}|\X u|^4\, Tu$$ 
as a testing function for  equation 
\eqref{equation:T}. We obtain that 
\begin{equation}\label{est:Ki}
 \begin{aligned}
 &\int_\Omega \sum_{i,j=1}^{2n}\eta^{\tau(\beta+2)+4}\,v^{\tau(\beta+4)} 
|\X u|^4 D_jD_if(\X u)X_jTu\,X_iTu\, dx\\
=&-(\tau(\beta+2)+4)\int_\Omega \sum_{i,j=1}^{2n} \eta^{\tau(\beta+2)+3}\,v^{\tau(\beta+4)} 
|\X u|^4 Tu D_jD_if(\X u)X_jTu\,X_i\eta\dx\\
&-\tau(\beta+4)\int_\Omega \sum_{i,j=1}^{2n}\eta^{\tau(\beta+2)+4}
\,v^{\tau(\beta+4)-1} 
|\X u|^4 Tu D_jD_if(\X u)X_jTu\,X_i v\dx\\
&-4\int_\Omega \sum_{i,j,k=1}^{2n} \eta^{\tau(\beta+2)+4}\,v^{\tau(\beta+4)} 
|\X u|^2 X_ku Tu D_jD_if(\X u)X_jTu\,X_iX_ku\dx\\
=& K_1 +K_2+ K_3,
\end{aligned}
\end{equation}
where the integrals in the right hand side of \eqref{est:Ki} are denoted by
$K_1,K_2,K_3$ in order, respectively.
We estimate both sides of (\ref{est:Ki}) as follows. For the left hand side, 
we have by the structure condition \eqref{structure} that
\begin{equation}\label{est:K}
 \text{left of \eqref{est:Ki}} \geq \int_\Omega \eta^{\tau(\beta+2)+4}\,v^{\tau(\beta+4)} 
\weight |\X u|^4|\X (Tu)|^2\, dx = M.
\end{equation}

For the right hand side of (\ref{est:Ki}), we estimate each item $K_i, i=1,2,3$, one by one. To this end, 
we denote  
\begin{equation}\label{def:tilde K}
\tilde K = \int_\Om \eta^{(2\tau-1)(\beta+2)+6} 
\,v^{(2\tau-1)(\beta+4)} \weight |\X u|^4 |Tu|^2 |\X(Tu)|^2\dx.
\end{equation} 
First, we estimate $K_1$ by the structure condition \eqref{structure} and 
H\"older's inequality. We have
\begin{equation}\label{est:K1}
\begin{aligned}
\vert K_1\vert \leq & c(\beta+2)\int_\Omega \eta^{\tau(\beta+2)+3}\,v^{\tau(\beta+4)} 
\weight |\X u|^4|Tu||\X (Tu)||\X\eta|\, dx\\
\leq & c(\beta+2) \tilde K^\frac{1}{2} 
\Big(\int_\Om \eta^{\beta+2} v^{\beta+4} 
\weight |\X u|^4 |\X\eta|^2\dx\Big)^\frac{1}{2},
\end{aligned}
\end{equation}
where $c = c(n,p,L,\tau)>0$. 

Second, we estimate $K_2$ also by the structure condition \eqref{structure} and 
H\"older's inequality. We have 
\begin{equation}\label{est:K2}
\begin{aligned}
 \vert K_2\vert \leq & c(\beta+2)\int_\Omega \eta^{\tau(\beta+2)+4}\,v^{\tau(\beta+4)-1} 
\weight |\X u|^4|Tu||\X (Tu)||\X v|dx\\
 \leq & c(\beta+2) \tilde K^\frac{1}{2} 
\Big(\int_\Om \eta^{\beta+4} v^{\beta+2} 
\weight |\X u|^4 |\X v|^2\dx\Big)^\frac{1}{2}.
\end{aligned}
\end{equation}

Finally, we estimate $K_3$. In the following, the first inequality
follows from the
structure condition \eqref{structure}, the second from 
H\"older's inequality and the third from Lemma \ref{lem:start}. We have 
\begin{equation}\label{est:K3}
\begin{aligned}
\vert K_3\vert &\leq c\int_\Omega \eta^{\tau(\beta+2)+4}\,v^{\tau(\beta+4)} 
\weight |\X u|^3|Tu||\X (Tu)||\X\X u|\, dx\\
&\leq c \tilde K^\frac{1}{2} 
\Big(\int_\Om \eta^{\beta+4} v^{\beta+4} 
\weight |\X u|^2 |\X\X u|^2\dx\Big)^\frac{1}{2}\\
&\leq c\, \tilde K^\frac{1}{2}\, I^\frac{1}{2}, 
\end{aligned}
\end{equation}
where $I$ is the right hand side of \eqref{ineq1} in Lemma \ref{lem:start}
\begin{equation}\label{def:I}
\begin{aligned}
I=&  c (\beta+2)^2\int_\Omega \eta^{\beta+2}
v^{\beta+4}
\weight |\X u|^4\big(|\X \eta|^2+\eta| T\eta|\big)\, dx\\
& +  c(\beta+2)^2\int_\Omega \eta^{\beta+4} v^{\beta+2}\weight |\X u|^4
| \X v|^2\, dx\\
& +
c\int_\Omega\eta^{\beta+4} v^{\beta+4}\weight |\X u|^2 | Tu|^2\, dx. 
\end{aligned}
\end{equation}
and $c=c(n,p,L)>0$.
Notice that the integrals on the right hand side of \eqref{est:K1} and \eqref{est:K2} are both controlled from above by $I$. 
Hence, we can combine \eqref{est:K1}, \eqref{est:K2} and \eqref{est:K3} to obtain that 
\[\vert K_1\vert+\vert K_2\vert+\vert K_3\vert \leq\ c \tilde K^\frac{1}{2} I^\frac{1}{2}, 
\]
from which,  together with the estimate \eqref{est:K} for the left hand side of (\ref{est:Ki}), it follows that
\begin{equation}\label{est:K'}
M\le c {\tilde K}^\frac{1}{2} I^\frac{1}{2},
\end{equation}
where $c = c(n,p,L,\tau)>0$. 
Now, we estimate $\tilde K$ by H\"older's inequality as follows. 
\begin{equation}\label{est:tilde K}
 \begin{aligned}
  \tilde K\ \leq & \Big(\int_\Omega\eta^{\tau(\beta+2)+4}\,v^{\tau(\beta+4)} 
\weight |\X u|^4|\X (Tu)|^2\, dx\Big)^{\frac{2\tau-1}{\tau}}\\
& \times\Big(\int_\Omega\eta^{\frac{2\tau}{1-\tau}+4}
\weight |\X u|^4 |Tu|^{\frac{2\tau}{1-\tau}}
|\X(Tu)|^2\dx\Big)^\frac{1-\tau}{\tau}\\
=& M^{\frac{2\tau-1}{\tau}}\, G^\frac{1-\tau}{\tau}, 
 \end{aligned}
\end{equation}
where $M$ is as in \eqref{def:K} and we denote by $G$ the second integral on the right hand side of (\ref{est:tilde K})
\begin{equation}\label{def:M}
 G=\int_\Omega\eta^{\frac{2\tau}{1-\tau}+4}
\weight |\X u|^4 |Tu|^{\frac{2\tau}{1-\tau}}
|\X(Tu)|^2\dx.
\end{equation}
Now \eqref{est:tilde K} and  \eqref{est:K'} yield that
\begin{equation}\label{est:K''}
 M \leq c G^{1-\tau} I^{\tau},
\end{equation}
where $c = c(n,p,L,\tau)>0$. 
To estimate $K$, we estimate $G$ and $I$ from above. 
We estimate $G$ by Corollary \ref{cor:TX} with $q = 2/(1-\tau)$, and we  obtain
that
\begin{equation}\label{est:M}
\begin{aligned}
G\leq & c\mu(r)^4\int_\Omega \eta^{q+2} \weight | Tu|^{q-2}|\X(Tu)|^2\, dx\\
\le & \frac{c}{r^{q+2}}\mu(r)^4\int_{B_r}\weight |\X u|^q \, dx\\
 \le & \frac{c}{r^{q+2}}| B_r|\big(\delta+\mu(r)^2\big)^{\frac{p-2}{2}}
 \mu(r)^{q+4},
\end{aligned}
\end{equation}
where $c = c(n,p,L,\tau)>0$ and in the last inequality we used (\ref{trivial}). 

Now, we fix 
$1<\gamma<2$ and estimate each term of $I$ in \eqref{def:I} as follows. 
For the first term of $I$, we have by H\"older's inequality and (\ref{trivial}) that 
\begin{equation}\label{I1}
\begin{aligned}
\int_\Omega \eta^{\beta+2}
v^{\beta+4}
&\weight |\X u|^4\big(|\X \eta|^2+\eta| T\eta|\big)\, dx\\
& \leq\ \frac{c}{r^2}\big(\delta+\mu(r)^2\big)^{\frac{p-2}{2}}\mu(r)^8
|B_r|^{1-\frac{1}{\gamma}}
\Big(\int_{B_r} \eta^{\gamma\beta}v^{\gamma\beta}\dx\Big)^\frac{1}{\gamma}.
\end{aligned}
\end{equation}
For the second term of $I$, we have by (\ref{trivial}) that 
\begin{equation}\label{I2}
\begin{aligned}
\int_\Omega \eta^{\beta+4} v^{\beta+2}&\weight |\X u|^4
| \X v|^2\, dx\\
&\leq  c\big(\delta+\mu(r)^2\big)^{\frac{p-2}{2}}\mu(r)^4
\int_{B_r} \eta^{\beta+4} v^{\beta+2}| \X v|^2\, dx
\end{aligned}
\end{equation}
For the third term of $I$, we have that
\begin{equation}\label{I3}
\begin{aligned}
\int_\Omega 
\eta^{\beta+4} v^{\beta+4}
&\weight |\X u|^2|Tu|^2\dx\\
\ \leq\ & \Big(\int_\Om \eta^\frac{2\gamma}{\gamma-1}\weight |\X u|^2
|Tu|^\frac{2\gamma}{\gamma-1}\dx\Big)^{1-\frac{1}{\gamma}}\\
&\times\Big( \int_\Om \eta^{\gamma(\beta+2)}v^{\gamma(\beta+4)}
\weight |\X u|^2 \dx\Big)^\frac{1}{\gamma}\\
\ \leq\ & \frac{c}{r^2}\big(\delta+\mu(r)^2\big)^{\frac{p-2}{2}}\mu(r)^8
|B_r|^{1-\frac{1}{\gamma}}
\Big(\int_{B_r} \eta^{\gamma\beta}v^{\gamma\beta}\dx\Big)^\frac{1}{\gamma}
\end{aligned}
\end{equation}
where $c = c(n,p,L,\gamma)>0$. Here in the above inequalities, the first one follows from H\"older's inequality and the second from Lemma \ref{cor:Tu:high} 
and (\ref{trivial}).  
Therefore, the estimates for three items of $I$ above \eqref{I1}, \eqref{I2} and \eqref{I3} give us the following one for $I$
 \begin{equation}\label{est:I}
\begin{aligned}
I \leq c(\beta+2)^2 \big(\delta+\mu(r)^2\big)^{\frac{p-2}{2}}\mu(r)^4 J,
\end{aligned}
\end{equation}
where $J$ is defined as in \eqref{def:J}   
\begin{equation*}
 J = \int_{B_r} \eta^{\beta+4} v^{\beta+2} |\X v|^2\dx 
\ +\ \mu(r)^4\frac{|B_r|^{1-\frac{1}{\gamma}}}{r^2} 
\Big(\int_{B_r} \eta^{\gamma\beta}v^{\gamma\beta}\dx\Big)^\frac{1}{\gamma}.
\end{equation*}
Now from  the estimates \eqref{est:M} for $G$ and  
\eqref{est:I} for $I$, we obtain the desired estimate for $M$ by \eqref{est:K''}. Combing \eqref{est:M}, \eqref{est:I} and \eqref{est:K''}, we end up with 
\begin{equation}\label{est:K f}
M \leq c (\beta+2)^{2\tau}\frac{|B_r|^{1-\tau}}{r^{2(2-\tau)}}\big(\delta+\mu(r)^2\big)^{\frac{p-2}{2}}
 \mu(r)^6 J^{\tau},
\end{equation}
where $c = c(n,p,L,\tau,\gamma)>0$. This completes the proof. 
\end{proof}
Now we prove the main lemma, Lemma \ref{lemma:main}. 
We restate Lemma \ref{lemma:main} here.
\begin{lemma}\label{lemma:mainprime}
Let $\gamma>1$ be a number and for $B_r\subset\Omega$, $\eta\in C^\infty_0(B_r)$, be a cut-off function satisfying (\ref{etaBr1}) and (\ref{etaBr2}). We have the following Caccioppoli type inequality
 \begin{equation}\label{C_v1}
  \int_{B_r} \eta^{\beta+4}v^{\beta+2}|\X v|^2\, dx\ \leq\ c(\beta+2)^2\mu(r)^4
\frac{| B_r|^{1-1/\gamma}}{r^2}
  \Big(\int_{B_r}\eta^{\gamma\beta}v^{\gamma\beta}\, dx\Big)^{1/\gamma}
\end{equation}
for all $\beta\geq0$, where $c=c(n,p, L,\gamma)>0$.
\end{lemma}

\begin{proof} We note that we may assume that $\gamma<3/2$, since otherwise we can apply H\"older's inequality to the integral in the right hand side of the claimed inequality (\ref{C_v1}). So, we fix $ 1<\gamma < 3/2$.  
Recall that
\[v=\min\big(\mu(r)/8,\max(\mu(r)/4-X_lu,0)\big),\]
where $l\in\{ 1,2,...,2n\}$. We only prove the lemma for $l\in\{ 1,2,...,n\}$; we can prove the lemma similarly for $l\in \{ n+1,n+2,...,2n\}$. 
Now fix $l\in \{ 1,2,...,n\}$. Let $\beta\ge 0$ and $\eta\in C^\infty_0(B_r)$ be a cut-off function satisfying (\ref{etaBr1}) and (\ref{etaBr2}). We use 
\[ \varphi=\eta^{\beta+4}v^{\beta+3} \]
as a test function for equation (\ref{equation:horizontal}) to obtain that
\begin{equation}\label{lemest1}
\begin{aligned}
-\int_\Omega \sum_{i,j=1}^{2n} D_jD_if(\X u)X_jX_luX_i\varphi\, dx= & \int_\Omega \sum_{i=1}^{2n} D_{n+l}D_if(\X u)Tu X_i\varphi\, dx\\
& -\int_\Omega T\big( D_{n+l}f(\X u)\big) \varphi\, dx.
\end{aligned}
\end{equation}
Note that 
$$X_i \varphi=(\beta+3)\eta^{\beta+4}v^{\beta+2} X_i v+(\beta+4)\eta^{\beta+3}v^{\beta+3}X_i \eta.$$ 
Thus (\ref{lemest1}) becomes
\begin{equation}\label{lemest2}
\begin{aligned}
-(\beta+3)&\int_\Omega \sum_{i,j=1}^{2n}\eta^{\beta+4}v^{\beta+2} D_jD_if(\X u)X_jX_luX_i v\, dx\\
\ =\ & (\beta+4)\int_\Omega \sum_{i,j=1}^{2n} \eta^{\beta+3}v^{\beta+3} D_jD_if(\X u) X_jX_luX_i\eta\, dx\\
& + (\beta+4)\int_\Omega\sum_{i=1}^{2n} \eta^{\beta+3}v^{\beta+3} D_{n+l}D_if(\X u)Tu X_i\eta\, dx\\
&+(\beta+3)\int_\Omega \sum_{i=1}^{2n} \eta^{\beta+4} v^{\beta+2} D_{n+l}D_if(\X u)X_i v \,Tu\, dx\\
&-\int_\Omega \eta^{\beta+4}v^{\beta+3}\,T\big( D_{n+l}f(\X u)\big)\, dx.
\end{aligned}
\end{equation}
Note that
\[ X_jX_l-X_lX_j=0, \quad \text{ if } j\neq n+l,\]
and that
\[ X_{n+l}X_l-X_lX_{n+l}=-T.\]
Therefore we have
\begin{equation*}
\begin{aligned}
&\sum_{i,j=1}^{2n}D_jD_if(\X u) X_jX_luX_i\eta+ \sum_{i=1}^{2n} D_{n+l}D_if(\X u)Tu X_i\eta\\ 
=&\sum_{i,j=1}^{2n}D_jD_if(\X u) X_lX_juX_i\eta=\sum_{i=1}^{2n} X_l\big(D_if(\X u)\big)X_i\eta.
\end{aligned}
\end{equation*}
Now we can combine the first two integrals in the right hand side of (\ref{lemest2}) by the above equality. Then (\ref{lemest2}) becomes
\begin{equation}\label{lemest3}
\begin{aligned}
-(\beta+3)&\int_\Omega \sum_{i,j=1}^{2n}\eta^{\beta+4}v^{\beta+2} D_jD_if(\X u)X_jX_luX_i v\, dx\\
= \ 
& (\beta+4)\int_\Omega\sum_{i=1}^{2n} \eta^{\beta+3}v^{\beta+3} X_l\big(D_if(\X u)\big)X_i\eta \, dx\\
&+(\beta+3)\int_\Omega \sum_{i=1}^{2n} \eta^{\beta+4} v^{\beta+2} D_{n+l}D_if(\X u)X_ivTu\, dx\\
& -\int_\Omega \eta^{\beta+4}v^{\beta+3}T\big( D_{n+l}f(\X u)\big)\, dx\\
=\  &\, I_1+I_2+I_3.
\end{aligned}
\end{equation}
Here we denote the terms in the right hand side of (\ref{lemest3}) by $I_1,I_2,I_3$, respectively. 

We will estimate both sides of (\ref{lemest3}) as follows. For the left hand side, we have by the structure condition (\ref{structure}) that
\begin{equation}\label{est:lemleft}
\begin{aligned}
\text{left of } (\ref{lemest3}) & \ge (\beta+3)\int_\Omega \eta^{\beta+4}
v^{\beta+2}\weight |\X v|^2\, dx\\
&\ \ge\ c_0(\beta+2)\big(\delta+\mu(r)^2\big)^{\frac{p-2}{2}}\int_{B_r}\eta^{\beta+4}v^{\beta+2}|\X v|^2\, dx,
\end{aligned}
\end{equation}
where $c_0=c_0(n,p, L)>0$. Here we used \eqref{vis0} and \eqref{comparable1}.

For the right hand side of (\ref{lemest3}), we claim that
each item $I_1,I_2,I_3$ satisfies the following estimate 
\begin{equation}\label{lem:claim}
\begin{aligned}
| I_m|\ \le &\ \frac{c_0}{6}(\beta+2)\big(\delta+\mu(r)^2\big)^{\frac{p-2}{2}}\int_{B_r}\eta^{\beta+4}v^{\beta+2}|\X v|^2\, dx\\
&\ +c(\beta+2)^3
\frac{| B_r|^{1-1/\gamma}}{r^2}
  \big(\delta+\mu(r)^2\big)^{\frac{p-2}{2}}\mu(r)^4\Big(\int_{B_r}\eta^{\gamma\beta}v^{\gamma\beta}\, dx\Big)^{1/\gamma},
\end{aligned}
\end{equation}
where $m=1,2,3$, $1<\gamma<3/2$ and  $c$  is a constant depending only on $n,p,L$ and $\gamma$.
Then the lemma follows from the estimate (\ref{est:lemleft}) for the
left hand side of (\ref{lemest3}) and the above claim
(\ref{lem:claim}) for each item in the right. This completes the proof of the lemma, modulo the proof of the claim (\ref{lem:claim}). 

In the rest of the proof, we estimate
$I_1, I_2,I_3$ one by one. 
First, for $I_1$, we have by integration by parts that
\[ I_1=-(\beta+4)\int_\Omega\sum_{i=1}^{2n} D_if(\X u)X_l\big( \eta^{\beta+3}
v^{\beta+3}X_i\eta\big)\, dx,\]
from which it follows by the structure condition (\ref{structure}) that
\begin{equation}\label{lemest4}
\begin{aligned}
| I_1|\le &\, c(\beta+2)^2\int_\Omega \eta^{\beta+2} v^{\beta+3}\weight |\X u|\big(|\X\eta|^2+\eta|\X\X\eta|\big)\, dx\\
&+ c(\beta+2)^2\int_\Omega \eta^{\beta+3}v^{\beta+2}\weight |\X u||\X v\|\X\eta|\, dx\\
\le & \, \frac{c}{r^2}(\beta+2)^2\big(\delta+\mu(r)^2\big)^{\frac{p-2}{2}}\mu(r)^4\int_{B_r} \eta^\beta
v^\beta\, dx\\
&\, +\frac{c}{r}(\beta+2)^2\big(\delta+\mu(r)^2\big)^{\frac{p-2}{2}}
\mu(r)^2\int_{B_r}\eta^{\beta+2}v^{\beta+1}|\X v|\, dx,
\end{aligned}
\end{equation}
where $c=c(n,p,L)>0$. Here the second inequality follows from (\ref{trivial}), from the definitions of $\mu(r)$ and $v$, and from the factor that the support of $\eta$ lies in $B_r$. 
Now we apply Young's inequality to the last term of inequality (\ref{lemest4}) to end up with the following estimate for $I_1$.
\begin{equation}\label{lemest:I1}
\begin{aligned}
| I_1|\le &\, \frac{c_0}{6} (\beta+2)\big(\delta+\mu(r)^2\big)^{\frac{p-2}{2}}
\int_{B_r}\eta^{\beta+4}v^{\beta+2}|\X v|^2\, dx\\
&\, +\frac{c}{r^2}(\beta+2)^3\big(\delta+\mu(r)^2\big)^{\frac{p-2}{2}}\mu(r)^4\int_{B_r}\eta^\beta v^\beta\, dx,
\end{aligned}
\end{equation}
where $c=c(n,p,L)>0$ and $c_0$ is the same constant as in (\ref{est:lemleft}).
Now the claimed estimate (\ref{lem:claim}) for $I_1$ follows from the above estimate (\ref{lemest:I1}) and H\"older's inequality. 
 
Second, to estimate $I_2$, we have by the structure condition (\ref{structure}) that
\[ | I_2| \le c(\beta+2) \int_\Omega \eta^{\beta+4} v^{\beta+2} \weight 
| \X v| | Tu|\, dx,
\]
from which it follows by H\"older's inequality that
\begin{equation}\label{lemest5}
\begin{aligned}
| I_2| \le  &\, c(\beta+2) \Big(\int_E \eta^{\beta+4}v^{\beta+2}\weight|\X v|^2\, dx\Big)^{\frac{1}{2}}\\
&\qquad\quad \times \Big(\int_E \eta^{\gamma(\beta+2)}v^{\gamma(\beta+2)}\weight \, dx\Big)^{\frac{1}{2\gamma}}\\
& \qquad \quad \times\Big(\int_\Omega \eta^q \weight| Tu|^q\, dx\Big)^{\frac{1}{q}},
\end{aligned}
\end{equation}
where $q=2\gamma/(\gamma-1)$. 
Here we used (\ref{vis0}) so that in the second integral we can put the integration domain to be the set $E$, defined as in (\ref{setE}). This is critical, otherwise 
we would not have estimate for this integral and for the first integral in the case $1<p<2$. But now in set 
$E$ we have (\ref{comparable1}), and we have the following estimates for these two integrals for the full range $1<p<\infty$.
\begin{align}
\label{lemest6}
\int_E \eta^{\beta+4}v^{\beta+2}\weight|\X v|^2\, dx
 \le c \big(\delta+\mu(r)^2\big)^{\frac{p-2}{2}} \int_{B_r} \eta^{\beta+4}
v^{\beta+2}|\X v|^2\, dx, 
\end{align}
and 
\begin{align}
\label{lemest7}
\int_E \eta^{\gamma(\beta+2)}v^{\gamma(\beta+2)}\weight \, dx\le c \big(\delta+\mu(r)^2\big)^{\frac{p-2}{2}}\mu(r)^{2\gamma}\int_{B_r} \eta^{\gamma\beta} v^{\gamma\beta}\, dx,
\end{align}
where $c=c(n,p)>0$. We estimate the last integral in the right hand side of 
(\ref{lemest5}) by Lemma \ref{cor:Tu:high}. We have
\begin{equation}\label{lemest8}
\begin{aligned}
\int_\Omega \eta^q \weight| Tu|^q\, dx& \le \frac{c}{r^q}\int_{B_r}\weight |\X u|^q\, dx\\
&\le \frac{c| B_r|}{r^q}\big(\delta+\mu(r)^2\big)^{\frac{p-2}{2}}\mu(r)^q,
\end{aligned}
\end{equation}
where $c=c(n,p,L, \gamma)>0$. Here we used (\ref{trivial}) again.  Now combining the above three estimates (\ref{lemest6}), (\ref{lemest7}) and (\ref{lemest8}) for the three integrals in (\ref{lemest5}) respectively, we end up with the following estimate for $I_2$ 
\begin{equation*}
| I_2| \le c(\beta+2)\frac{| B_r|^{\frac{1}{q}}}{r}\big(\delta+\mu(r)^2\big)^{\frac{p-2}{2}}
\mu(r)^2\Big(\int_{B_r}\eta^{\beta+4}v^{\beta+2}|\X v|^2\, dx\Big)^{\frac{1}{2}}
\Big(\int_{B_r}\eta^{\gamma\beta}v^{\gamma\beta}\, dx\Big)^{\frac{1}{2\gamma}},
\end{equation*}
from which, together with Young's inequality, the claim 
(\ref{lem:claim}) for $I_2$ follows.

Finally, we prove (\ref{lem:claim}) for $I_3$.  Recall that
\[ I_3=-\int_\Omega\eta^{\beta+4}v^{\beta+3}T\big(D_{n+l}f(\X u)\big)\, dx.\]
Due to the regularity (\ref{reg:v}) for $v$,  integration by parts yields
\begin{equation}\label{lemest10}
\begin{aligned}
I_3 =&\, \int_\Omega D_{n+l}f(\X u) T\big(\eta^{\beta+4}v^{\beta+3}\big)\, dx\\
=&\, (\beta+4)\int_\Omega \eta^{\beta+3}v^{\beta+3} D_{n+l}f(\X u)T\eta\, dx\\
&\, +(\beta+3)\int_\Omega \eta^{\beta+4}v^{\beta+2} D_{n+l}f(\X u)Tv\, dx
= I_3^1+I_3^2,
\end{aligned}
\end{equation}
where we denote the last two integrals in the above equality by $I_3^1$ and $I_3^2$, respectively. The estimate for $I_3^1$ is easy. By the structure condition (\ref{structure}) and by (\ref{trivial}), we have
\begin{equation}\label{est:I31}
\begin{aligned}
| I_3^1|\le & \, c(\beta+2) \int_\Omega \eta^{\beta+3}v^{\beta+3}
\weight |\X u|| T\eta|\, dx\\
\le & \,\frac{c}{r^2}\big(\delta+\mu(r)^2\big)^{\frac{p-2}{2}}\mu(r)^4\int_{B_r}\eta^\beta
v^\beta\, dx.
\end{aligned}
\end{equation}
Thus by H\"older's inequality, $I_3^1$ satisfies estimate (\ref{lem:claim}). Now we  estimate $I_3^2$. We note that by (\ref{vis0}) and the 
structure condition (\ref{structure}) we have
\begin{equation}\label{lemest11}
| I_3^2|\le c(\beta+2)\int_E \eta^{\beta+4}v^{\beta+2}\weight |\X u|| \X (Tu)|\, dx,
\end{equation}
where the set $E$ is 
\[ E=\{ x\in \Omega: \mu(r)/8< X_lu<\mu(r)/4\},\]
defined as in (\ref{setE}). We continue to estimate $I_3^2$ by H\"older's inequality
\[\begin{aligned} | I_3^2|\le c(\beta+2)&\Big(\int_E \eta^{(2-\gamma)(\beta+2)+4}v^{(2-\gamma)(\beta+4)} \weight |\X u|^2|\X (Tu)|^2\, dx\Big)^{\frac 1 2}\\
& \times \Big(\int_E\eta^{\gamma(\beta+2)}v^{\gamma\beta+4(\gamma-1)}\weight\, dx\Big)^{\frac 1 2}.\end{aligned}\]
We remark that in set $E$ we have (\ref{comparable1}). Thus 
\begin{equation}\label{est:I32}
| I_3^2| \le c(\beta+2)\big(\delta+\mu(r)^2\big)^{\frac{p-2}{4}}\mu(r)^{2(\gamma-1)-1} M^{\frac 1 2}\Big(\int_{B_r} \eta^{\gamma\beta}
v^{\gamma\beta}\, dx\Big)^{\frac 1 2},
\end{equation}
where 
\begin{equation}\label{def:K new}
 M= \int_\Omega \eta^{(2-\gamma)(\beta+2)+4}\,v^{(2-\gamma)(\beta+4)} 
\weight |\X u|^4|\X (Tu)|^2\, dx.
\end{equation}
Now we are in a position to apply Lemma \ref{lem:tech} to estimate 
$M$ from above. 
Lemma \ref{lem:tech}
with $\tau = 2-\gamma$ gives us that
\begin{equation}\label{est:K final}
M \leq c (\beta+2)^{2(2-\gamma)}\frac{|B_r|^{\gamma-1}}{r^{2\gamma}}\big(\delta+\mu(r)^2\big)^{\frac{p-2}{2}}
 \mu(r)^6\,J^{2-\gamma}
\end{equation}
where $c = c(n,p,L,\gamma)>0$ and
$J$ is defined  as in  \eqref{def:J}     
\begin{equation}\label{def:J new}
 J = \int_{B_r} \eta^{\beta+4}v^{\beta+2} |\X v|^2\dx 
+ \mu(r)^4\frac{|B_r|^{1-\frac{1}{\gamma}}}{r^2} 
\Big(\int_{B_r} \eta^{\gamma\beta}v^{\gamma\beta}\dx\Big)^\frac{1}{\gamma}.
\end{equation}
Now, it follows from \eqref{est:K final} and  \eqref{est:I32} that 
$$ |I^2_3|\leq c(\beta+2)^{3-\gamma} 
\big(\delta+\mu(r)^2\big)^{\frac{p-2}{2}}\mu(r)^{2\gamma}\,
\frac{|B_r|^\frac{\gamma-1}{2}}{r^\gamma}\,
J^\frac{2-\gamma}{2}
\Big(\int_{B_r} 
\eta^{\gamma\beta}v^{\gamma\beta}\dx\Big)^\frac{1}{2}.$$
By Young's inequality, we end up with
\begin{equation*}
\begin{aligned}
|I^2_3|\leq & \frac{c_0}{12}(\beta+2)\big(\delta+\mu(r)^2\big)^{\frac{p-2}{2}} J\\
 & + c\,(\beta+2)^{\frac{4}{\gamma}-1}
 \big(\delta+\mu(r)^2\big)^{\frac{p-2}{2}}\mu(r)^4
\frac{|B_r|^{1-\frac{1}{\gamma}}}{r^2}
\Big(\int_{B_r} \eta^{\gamma\beta}v^{\gamma\beta}\dx\Big)^\frac{1}{\gamma},
\end{aligned}
\end{equation*}
where $c_0>0$ is the same constant as in \eqref{lem:claim}. 
Note that $ J$ is defined in \eqref{def:J new}. Thus $I_3^2$ satisfies 
a similar estimate to \eqref{lem:claim}. Now the desired claim \eqref{lem:claim}
for $I_3$ follows, since both $I_3^1$ and $I_3^2$ satisfy similar estimates.  
This concludes the proof of the claim \eqref{lem:claim}, and hence the proof of the lemma.
\end{proof}

\begin{remark}\label{remark1}
We can prove in the same way as that of Lemma \ref{lemma:main} that the conclusion (\ref{C_v1}) holds for
\[ v^\prime=\min\big(\mu(r)/8, \max(\mu(r)/4+X_lu,0)\big).\]
\end{remark}

The following corollary follows from Lemma \ref{lemma:main} by Moser's iteration. It is proved for the case $p\ge 2$ in \cite{Zh}, see Lemma 4.4 of \cite{Zh}.  Its proof is standard and is the same as
in the Euclidean setting, see 
Proposition 4.1 of \cite{Di} or Lemma 2 of \cite{To}. We include the proof here.

\begin{corollary}\label{prop:case1}
There exists a constant $\theta = \theta(n,p,L)>0$ such that the following statements hold. If we have 
\begin{equation}\label{condition1}
 \vert\{x\in B_r : X_lu<\mu(r)/4\} \vert\leq \theta |B_r|
\end{equation}
for an index $l\in\{1,\ldots,2n\}$ and for a ball $B_r\subset\Omega$,  then 
 \[\inf_{B_{r/2}}X_lu\ge 3\mu(r)/16;\]
Analogously, if we have 
\begin{equation}\label{condition2}
\vert\{x\in B_r : X_lu>-\mu(r)/4\}\vert\leq \theta |B_r|,
\end{equation}
for an index $l\in\{1,\ldots,2n\}$ and for a ball $B_r\subset\Omega$,  then  
\[\sup_{B_{r/2}}X_lu \le\, -3\mu(r)/16.\]
\end{corollary}

\begin{proof}
Suppose that (\ref{condition1}) holds for an index $l\in \{1,2,...,2n\}$.  We will apply Lemma \ref{lemma:mainprime} to prove Corollary \ref{prop:case1}.
The case that (\ref{condition2}) holds can be handled similarly by Lemma \ref{lemma:mainprime} for the function $v^\prime$, see Remark \ref{remark1}.  

Let $\beta\geq 0$ and 
 $$w=\eta^{\beta/2+2}v^{\beta/2+2},$$ 
where $\eta\in C^\infty_0(B_r)$ is a cut-off function satisfying (\ref{etaBr1}) and (\ref{etaBr2})
and $v$ is defined as in \eqref{def:v}. Then for any $\gamma>1$, we have that
 \begin{equation}\label{prest1}
 \begin{aligned}
  \int_{B_r} |\X w|^2\dx
   \leq & c(\beta+2)^2\Big(\int_{B_r}\eta^{\beta+2} v^{\beta+4}|\X\eta|^2\dx 
  + \int_{B_r}\eta^{\beta+4}v^{\beta+2}|\X v|^2\dx\Big) \\
  \leq & c(\beta+2)^4\mu(r)^4\frac{|B_r|^{1-\frac{1}{\gamma}}}{r^2}
\Big(\int_{B_r} \eta^{\gamma\beta}v^{\gamma\beta}\dx\Big)^\frac{1}{\gamma},
 \end{aligned}
 \end{equation}
where $c = c(n,p,L,\gamma)>0$. 
Here the second inequality follows from H\"older's inequality and Lemma \ref{lemma:mainprime}. 
By the 
Sobolev inequality \eqref{sobolev}, we also have that
\begin{equation}\label{prest2}
\Big(\intav_{B_r} |w|^{2\chi}\dx\Big)^\frac{1}{\chi}\ \leq\ 
c(n)\,r^2 \intav_{B_r} |\X w|^2\dx,
\end{equation}
where $\chi = Q/(Q-2) = (n+1)/n$. Combining \eqref{prest1} and \eqref{prest2}, we obtain that
\begin{equation}\label{prest3}
\Big(\intav_{B_r} (\eta v)^{\chi(\beta+4)}\dx\Big)^\frac{1}{\chi}
\ \leq\ c(\beta+2)^4\, \mu(r)^4
\Big(\intav_{B_r} (\eta v)^{\gamma\beta}\dx\Big)^\frac{1}{\gamma},
\end{equation}
where $c = c(n,p,L,\gamma)>0$. 
Now, we choose $\gamma = (n+2)/(n+1)$. Thus
$1<\gamma < \chi$. 
We will iterate inequality (\ref{prest3}). Let
\[ \beta_i= \frac{4\chi}{\chi-\gamma}\big(\big(\frac{\chi}{\gamma}\big)^{i+1}-1\big),
\quad i= 0,1,2,\ldots .\] 
Note that $\gamma\beta_{i+1} = \chi(\beta_i +4)$. Thus 
\eqref{prest3} with $\beta = \beta_i$ becomes
\begin{equation}\label{itm}
M_{i+1} \leq c_i
M_{i}^{\frac{\chi}{\gamma}\frac{\beta_i}{\beta_{i+1}}}
\end{equation}
for every $i= 0,1,2,\ldots$, where 
\[ c_i= c^{\frac{\chi}{\gamma}\frac{1}{\beta_{i+1}}}\beta_{i+1}^{\frac{4\chi}{\gamma}{\frac{1}{\beta_{i+1}}}},\]
and
\[ M_i\,  = \left(\,\intav_{B_r} 
\big( \eta v/\mu(r)\big)^{\gamma\beta_i}\dx\right)^\frac{1}{\gamma\beta_i}.\]

Iterating \eqref{itm}, we obtain that 
\begin{equation}\label{itm'}
 M_i \leq\ c M_0^{\big(\frac{\chi}{\gamma}\big)^i\,\frac{\beta_0}{\beta_{i}}},
\end{equation}
where $c = c(n,p,L)>0$.
Let $i\to \infty$, we end up with 
\[\lim\sup_{i\to \infty} M_i \leq c\,M_0^{1-\gamma/\chi},\] 
that is,   
\begin{equation}\label{prest4}
\sup_{B_{r}}\, \eta v/\mu(r)\leq c 
 \Big(\intav_{B_r} \big(\eta v/\mu(r)\big)^{4\chi}\dx\Big)^{\frac{1}{4\chi}(1-\gamma/\chi)},
\end{equation}
where $c = c(n,p,L)>0$. Now, since $\eta$ satisfies \eqref{etaBr1} and (\ref{etaBr2}), we derive from (\ref{prest4}) by our assumption (\ref{condition1}) that
\[ \sup_{B_{r/2}}\, v
\leq c\mu(r)\,\theta^{\frac{1}{4\chi}(1-\gamma/\chi)}\le \mu(r)/16, \]
provided that $\theta$ is small enough. This implies that 
$X_l u\ge 3\mu(r)/16$ in $B_{r/2}$. The proof is finished.
\end{proof}

\section{H\"older continuity of the horizontal gradient\label{section:Holder}}
In this section, we prove Theorem \ref{thm:main}. This proof is divided into two cases, 
$\delta >0$ and $\delta =0$,
in subsection \ref{subsection:proof} and subsection \ref{subsection:proof0}, respectively.  The proof for the case $\delta>0$ is the same as that of Theorem 1.2 of \cite{Zh}, with minor modifications. The proof for the case $\delta=0$ follows from an approximation arguments, see \cite{Zh}. We include the proof here.

\subsection{Proof of Theorem
\ref{thm:main} for the case $\delta >0$}\label{subsection:proof}
Let $u\in HW^{1,p}(\Omega)$ be a weak solution of equation (\ref{equation:main}) satisfying the structure condition (\ref{structure}) with $\delta>0$.
We fix a ball $B_{r_0}\subset\Omega$. For all balls
 $B_r, 0<r<r_0$, with the same center as $B_{r_0}$, we denote for $l=1,2,...,2n,$
\[ \mu_l(r)=\sup_{B_r} \vert X_lu\vert, \quad \mu(r)=\max_{1\le l\le 2n}\mu_l(r),\]
and 
\[ \omega_l(r)=\osc_{B_r} X_lu, \quad \omega(r)=\max_{1\le l\le 2n}\omega_l(r).\]
Clearly, we have $\omega(r)\le 2\mu(r)$.

We define  for any function $w$
$$ A_{k,\rho}^+(w) =\{ x\in B_\rho: (w(x)-k)^+=\max (w(x)-k,0)>0\};$$ 
and we define $A_{k,\rho}^-(w)$ similarly.  
To prove Theorem \ref{thm:main}, we need the following lemma. 

\begin{lemma}\label{lemma:cacci:k}
Let $B_{r_0}\subset\Omega$ be a  ball and $0<r<r_0/2$. Suppose that there is $\tau>0$ such that 
 \begin{equation}\label{comparable'}
 \vert \X u\vert\ge \tau \mu(r) \quad \text{in }\, A_{k,r}^+(X_l u)
 \end{equation}
 for an index $l\in\{1,2,...,2n\}$ and for a constant $k\in \R$. Then for any $q\ge 4$ and any $0<r^{\prime\prime}<r^\prime\le r$, we have 
\begin{equation}\label{HDG}
\begin{aligned}
 &\int_{B_{r^{\prime\prime}}} \weight | \X(X_lu-k)^{+}|^2\, dx\\  \le & 
 \frac{c}{(r^\prime-r^{\prime\prime})^2}\int_{B_{r^{\prime}}} \weight
 |(X_lu-k)^{+}|^2\, dx+ cK| A^+_{k,r^\prime}(X_lu)|^{1-\frac{2} {q}}
\end{aligned}
\end{equation}
where $K = r_0^{-2}|B_{r_0}|^{2/q}
\big(\delta+\mu(r_0)^2\big)^{p/2}$ 
and $c\,=\,c(n,p,L,q, \tau)>0$. 
\end{lemma}

Lemma \ref{lemma:cacci:k} is similar to Lemma 4.3 of \cite{Zh}, which is valid for $p\ge 2$. Under our extra assumption (\ref{comparable'}), the proof of 
Lemma \ref{lemma:cacci:k} is exactly the same as that of Lemma 4.3 of \cite{Zh}. All of the steps go through in the same way. We  remark here that
there are two places in the proof of Lemma 4.3 of \cite{Zh} where the assumption $p\ge 2$ is used. Now due to our assumption (\ref{comparable'}), we may get through the proof for $1<p<\infty$. We omit the details of the proof of Lemma \ref{lemma:cacci:k}. 

\begin{remark}\label{rem:-version}
 Similarly, we can obtain an inequality, corresponding to (\ref{HDG}), 
 with 
 $(X_lu-k)^+$ replaced by $(X_lu-k)^-$ and
 $A^+_{k,r}(X_lu)$ replaced by $A^-_{k,r}(X_lu)$.
\end{remark}

Theorem \ref{thm:main} follows easily from the following theorem by an interation argument. 

\begin{theorem}\label{thm:holder2}
There exists a constant $s=s(n,p,L)\ge 1$ such that
for every $0<r\leq r_0/16$, we have 
\begin{equation}\label{mur4r}
\omega(r) \le (1-2^{-s})\omega(8r) + 2^s\big(\delta +\mu(r_0)^2\big)^\frac{1}{ 2}\left(\frac{r}{r_0}\right)^\alpha,
\end{equation}
where $\alpha = 1/2$ when $1<p<2$ and $\alpha = 1/p$ when $p\geq 2$. 
\end{theorem}

\begin{proof}
To prove Theorem \ref{thm:holder2}, 
we fix a ball $B_r$, with the same center as $B_{r_0}$, such that $0<r<r_0/16$. 
We may assume that
\begin{equation}\label{assum:mu}
\omega(r) \ge \big(\delta +\mu(r_0)^2\big)^\frac{1}{ 2}
\left(\frac{r}{r_0}\right)^{\alpha},
\end{equation}
since, otherwise, \eqref{mur4r} is true with $s=1$. 
In the following, we assume that (\ref{assum:mu}) is true, and we prove Theorem
\ref{thm:holder2}. 
We divide the 
proof of Theorem \ref{thm:holder2} into two cases.

{\it Case 1}. For at least one index $l\in\{1,\ldots,2n\}$, we have either
\begin{equation}\label{small1}
  |\{x\in B_{4r} : X_lu<\mu(4r)/4\}|
 \leq \theta |B_{4r}|
 \end{equation}
or
 \begin{equation}\label{small2}
  |\{x\in B_{4r}: X_lu>-\mu(4r)/4\}|
  \leq \theta |B_{4r}|, 
\end{equation}
where $\theta = \theta(n,p,L)>0$ is the constant in Corollary \ref{prop:case1}. 
Assume that (\ref{small1}) is true; the case (\ref{small2}) can be treated in the same way.
We apply Corollary \ref{prop:case1} and we obtain that 
\[\vert X_l u\vert \ge 3\mu(4r)/16\quad \text{in }\ B_{2r}.\]
Thus we have
 \begin{equation}\label{comparable2}
  \vert\X u\vert\ge  3\mu(2r)/16\quad \text{in }\ B_{2r}.
 \end{equation}
Due to (\ref{comparable2}),  we can apply Lemma \ref{lemma:cacci:k} with $q = 2Q$ to obtain that
\begin{equation}\label{DG estimate}
\begin{aligned}
\int_{B_{r''}} | \X (X_iu-k)^+|^2\, dx
\le & \frac{c}{(r'-r'')^2}\int_{B_{r'}}|(X_iu-k)^+|^2\,dx \\
&+  c K \big(\delta+\mu(2r)^2\big)^\frac{2-p}{2}| A^\pm_{k,r'}(X_iu)|^{1-\frac{1} {Q}}
\end{aligned}
\end{equation}
where $K = r_0^{-2}
|B_{r_0}|^{1/Q}\big(\delta+\mu(r_0)^2\big)^{p/2}$. 
The above inequality holds for all
$0<r''<r'\leq 2r,\ i\in\{1,\ldots,2n\}$ and all $k\in \R$. 
This means that for each $i$, $X_iu$ belongs to the De Giorgi class $DG^+(B_{2r})$, see Section 4.1 of \cite{Zh} for the definition. The corresponding version of 
Lemma \ref{lemma:cacci:k} for $(X_i u-k)^-$, see Remark \ref{rem:-version}, 
shows that $X_iu$ also belong to $DG^-(B_{2r})$. So, 
$X_iu$ belongs to $DG(B_{2r})$. Now we can apply Theorem 4.1 of \cite{Zh} to conclude that there is $s_0=s_0(n,p,L)>0$ such that for each $i\in\{1,2,...,2n\}$
\begin{equation}\label{osc}
 \osc_{B_{r}} X_iu\le (1-2^{-s_0})\osc_{B_{2r}}
X_iu + c K^\frac{1}{2}\big(\delta+\mu(2r)^2\big)^\frac{2-p}{4} r^\frac{1}{2}.
\end{equation}
Now notice that when $1<p<2$, we have that
$$\big(\delta+\mu(2r)^2\big)^\frac{2-p}{4} \leq \big(\delta+\mu(r_0)^2\big)^\frac{2-p}{4}.$$ 
When $p\geq 2$, our assumption \eqref{assum:mu} with $\alpha =1/p$ gives  
$$ \big(\delta+\mu(2r)^2\big)^\frac{2-p}{4}\leq 
2^{\frac{p-2}{2}}\omega(r)^\frac{2-p}{2} \leq 2^{\frac{p-2}{2}} \big(\delta+\mu(r_0)^2\big)^\frac{2-p}{4}
\left(\frac{r}{r_0}\right)^{\frac{2-p}{2p}},$$
where in the first inequality we used that $\mu(2r)\ge \omega(2r)/2\ge \omega(r)/2$.  
In both cases,  \eqref{osc} becomes
\begin{equation}\label{osc1}
 \osc_{B_{r}} X_iu \le (1-2^{-s_0})\osc_{B_{2r}}
X_iu+c\big(\delta +\mu(r_0)^2\big)^\frac{1}{ 2}\left(\frac{r}{r_0}\right)^\alpha,
\end{equation}
where $c=c(n,p,L)>0$,  $\alpha = 1/2$ when $1<p<2$ and $\alpha = 1/p$ when $p\geq 2$. This 
shows that in this case Theorem \ref{thm:holder2} is true. 

{\it Case 2}. If Case 1 does not happen, then for every $i\in\{1,\ldots,2n\}$, we have   
\begin{equation}\label{case2 1}
  |\{x\in B_{4r} : X_iu<\mu(4r)/4\}|
  > \theta |B_{4r}|,
 \end{equation} 
 and
 \begin{equation}\label{case2 2}
 |\{x\in B_{4r}: X_iu>-\mu(4r)/4\}|
 >\theta |B_{4r}|,
\end{equation}
where $\theta = \theta(n,p,L)>0$ is the constant in Corollary \ref{prop:case1}. 
Note that on the set 
$\{x\in B_{8r}: X_iu >\mu(8r)/4 \}$, we have trivially
 \begin{equation}\label{comparable4}
 \vert \X u\vert\ge \mu(8r)/4\quad \text{in } A^+_{k,8r}(X_iu)
 \end{equation}
 for all $k\geq \mu(8r)/4$. 
Thus, we can apply Lemma \ref{lemma:cacci:k} with $q = 2Q$ to conclude that 
\begin{equation}\label{DG estimate 2}
\begin{aligned}
\int_{B_{r''}} | \X (X_iu-k)^+|^2\, dx
\le& \frac{c}{(r'-r'')^2}\int_{B_{r'}}|(X_iu-k)^+|^2\,dx \\
& + c K\big(\delta+\mu(8r)^2\big)^\frac{2-p}{2}| A^+_{k,r'}(X_iu)|^{1-\frac{1} {Q}}
\end{aligned}
\end{equation}
where $K = r_0^{-2}
|B_{r_0}|^{1/Q}\big(\delta+\mu(r_0)^2\big)^{p/2}$, 
whenever $k\geq k_0= \mu(8r)/4 $ and $0<r^{\prime\prime}<r^\prime\le 8r$. 
The above inequality is true all $i\in\{ 1,2,...,2n\}$. We note that
\eqref{case2 1} implies trivially that
\[|\{x\in B_{4r} : X_iu<\mu(8r)/4\}|
  > \theta |B_{4r}|.\] 
Now we can apply Lemma 4.2 of  \cite{Zh} to conclude that there exists $s_1 = s_1(n,p,L)>0$ such that 
\begin{equation}\label{sup est}
 \sup_{B_{2r}} X_iu \le\sup_{B_{8r}} X_i u -2^{-s_1}\big(\sup_{B_{8r}}
X_iu-\mu(8r)/4\big)+ cK^\frac{1}{2}\big(\delta+\mu(8r)^2\big)^\frac{2-p}{4}r^\frac{1}{2}.
\end{equation}
From (\ref{case2 2}), we can derive similarly, see Remark \ref{rem:-version}, that
\begin{equation}\label{inf est}
 \inf_{B_{2r}} X_iu \ge \inf_{B_{8r}} X_iu +2^{-s_1}\big(-\inf_{B_{8r}}
X_iu-\mu(8r)/4\big) - c K^\frac{1}{2}\big(\delta+\mu(8r)^2\big)^\frac{2-p}{4}r^\frac{1}{2}.
\end{equation}
Note that the above two inequalities \eqref{sup est} and \eqref{inf est} yield
\[\osc_{B_{2r}} X_iu \le\ (1-2^{-s_1})\osc_{B_{8r}}
X_iu+ 2^{-s_1-1}\mu(8r)+c K^\frac{1}{2}\big(\delta+\mu(8r)^2\big)^\frac{2-p}{4}r^\frac{1}{2},
\]
and hence 
\begin{equation}\label{osc2}
\omega(2r) \leq \big(1-2^{-s_1}\big)\omega(8r)+2^{-s_1-1}\mu(8r)+c K^\frac{1}{2}\big(\delta+\mu(8r)^2\big)^\frac{2-p}{4}r^\frac{1}{2}.
\end{equation}
Now notice that when $1<p<2$, we have that
$$\big(\delta+\mu(8r)^2\big)^\frac{2-p}{4} \leq \big(\delta+\mu(r_0)^2\big)^\frac{2-p}{4}$$ 
When $p\geq 2$, our assumption \eqref{assum:mu} with $\alpha =1/p$ gives  
$$ \big(\delta+\mu(8r)^2\big)^\frac{2-p}{4}\leq 
2^{\frac{p-2}{2}}\mu(r)^\frac{2-p}{2} \leq 2^{\frac{p-2}{2}} \big(\delta+\mu(r_0)^2\big)^\frac{2-p}{4}
\left(\frac{r}{r_0}\right)^{\frac{2-p}{2p}},$$
where in the first inequality we used the fact that $\mu(8r)\ge \omega(8r)/2\ge \omega(r)/2$. 
In both cases,  \eqref{osc2} becomes
\[\omega(2r) \leq \big(1-2^{-s_1}\big)\omega(8r)+2^{-s_1-1}\mu(8r) +c\big(\delta +\mu(r_0)^2\big)^\frac{1}{ 2}\left(\frac{r}{r_0}\right)^\alpha.\]
Now we notice from the conditions (\ref{case2 1}) and (\ref{case2 2}) that
\[ \omega(8r)\ge \mu(8r)-\mu(4r)/4\ge 3\mu(8r)/4. \]
Then from the above two inequalities we arrive at
\[\omega(2r) \leq \big(1-2^{-s_1-2}\big)\omega(8r) +c\big(\delta +\mu(r_0)^2\big)^\frac{1}{ 2}\left(\frac{r}{r_0}\right)^\alpha,\]
where $c=c(n,p,L)>0$, $\alpha = 1/2$ when $1<p<2$ and $\alpha = 1/p$ when $p\geq 2$.
This shows that also in this case Theorem \ref{thm:holder2} is true. 
Thus, Theorem \ref{thm:holder2} is true with the  
choice of $s = \max(1, s_0, s_1+2, \log_2 c)$. The proof of Theorem \ref{thm:holder2} is finished. 
\end{proof}

\subsection{Proof of Theorem
\ref{thm:main} for the case $\delta =0$}\label{subsection:proof0}
The proof of Theorem \ref{thm:main} for this case follows from 
an approximation argument, exactly in the same way as
that in Section 5.3 of \cite{Zh}. Suppose that the integrand $f$ of functional
(\ref{functional}) satisfies the
structure condition 
\begin{equation}\label{structure 0}
\begin{aligned} | z|^{p-2} |\xi|^2\le
\langle D^2f(z) &\xi,\xi\rangle \le L|
z|^{p-2}| \xi|^2;\\
| Df(z)| & \le L| z|^{p-1}
\end{aligned}
\end{equation}
for all $z, \xi\in \R^{2n}$, where $L\ge 1$ is a constant.
We may assume that $f(0) = 0$. For $\delta>0$, we define
\begin{equation}\label{fdelta}
f_\delta(z)\ =\begin{cases} \big(\delta+f(z)^{\frac{2}{p}}\big)^{\frac p 2}, 
\quad &\text{ if } 1<p<2;\\
\delta^{\frac{p-2}{2}}|z|^2+f(z), &\text{ if } \,p\ge 2.
\end{cases}
\end{equation}
Then, it is easy to see that $f_\delta$ satisfies a structure condition similar to 
\eqref{structure} for all $\delta>0$, that is,    
\begin{equation}\label{structure4}
\begin{aligned} \frac{1}{\tilde L}(\delta+| z|^2)^{\frac{p-2}{2}} |\xi|^2 \le
 \langle D^2f_\delta(z) &\xi,\xi\rangle \le {\tilde L}(\delta+|
z|^2)^{\frac{p-2}{2}}| \xi|^2;\\
| Df_\delta(z)| & \le {\tilde L}  (\delta+| z|^2)^{\frac{p-2}{2}}|z|,
\end{aligned}
\end{equation}
where  $\tilde L = \tilde L(p,L)\geq 1$.
Now let $u\in HW^{1,p}(\Omega)$ be a solution of \eqref{equation:main}
satisfying the structure condition (\ref{structure 0}).   We 
denote by $u_\delta$ the unique weak solution of the following Dirichlet problem
\begin{equation}\label{direchelet3}
\begin{cases}
\divo \big(Df_\delta(\X w)\big) = 0 \quad  &\textrm{in } \Omega;\\
 w-u\in HW^{1,p}_0( \Omega).  &
\end{cases}
\end{equation}
Then we may apply Theorem \ref{thm:main} for the case $\delta>0$ to solution $u_\delta$. We obtain the uniform estimate (\ref{Xu:holder}) for $u_\delta$. 
Letting $\delta\to 0$, we conclude the proof of Theorem \ref{thm:main} for the case $\delta=0$. The proof is finished.

\section{Appendix}
\begin{proof}[Proof of Lemma \ref{lem:start}]
Fix $l\in\{ 1,2,...,n\}$ and $\beta\ge 0$. Let $\eta\in C^\infty_0(\Omega)$ be a non-negative cut-off function. Set
\begin{equation}\label{varphi} 
\varphi=\eta^{\beta+2}v^{\beta+2} |\X u|^2 X_lu.
\end{equation}
We use $\varphi$ as a test-function in equation (\ref{equation:horizontal}) to obtain that 
\begin{equation}\label{eq1}
\begin{aligned}
\int_\Omega &\sum_{i,j=1}^{2n} \eta^{\beta+2}v^{\beta+2}D_jD_if(\X u)X_jX_iuX_i\big( |\X u|^2 X_lu\big)\, dx\\
= & -(\beta+2)\int_\Omega\sum_{i,j=1}^{2n} \eta^{\beta+1} v^{\beta+2}|\X u|^2 X_lu D_jD_if(\X u)X_jX_luX_i\eta\, dx\\
& -(\beta+2)\int_\Omega \sum_{i,j=1}^{2n} \eta^{\beta+2} v^{\beta+1}
|\X u|^2 X_l uD_jD_if(\X u)X_iX_luX_iv\, dx\\
& -\int_\Omega\sum_{i=1}^{2n} D_{n+l}D_if(\X u) TuX_i\big(\eta^{\beta+2}
v^{\beta+2}|\X u|^2 X_lu\big)\, dx\\
& +\int_\Omega T\big(D_{n+l}f(\X u)\big)\eta^{\beta+2}v^{\beta+2}|\X u|^2 X_lu\, dx\\
= &\, I_1^l+I_2^l+I_3^l+I_4^l.
\end{aligned}
\end{equation}
Here we denote the integrals in the right hand side of (\ref{eq1}) by 
$I_1^l, I_2^l, I_3^l$ and $I_4^l$ in order respectively. Similarly, by equation
(\ref{equation:horizontal2}) we have for 
all $l\in\{n+1,n+2,...,2n\}$ that
\begin{equation}\label{eq2}
\begin{aligned}
\int_\Omega &\sum_{i,j=1}^{2n} \eta^{\beta+2}v^{\beta+2}D_jD_if(\X u)X_jX_iuX_i\big( |\X u|^2 X_lu\big)\, dx\\
= & -(\beta+2)\int_\Omega\sum_{i,j=1}^{2n} \eta^{\beta+1} v^{\beta+2}|\X u|^2 X_lu D_jD_if(\X u)X_jX_luX_i\eta\, dx\\
& -(\beta+2)\int_\Omega \sum_{i,j=1}^{2n} \eta^{\beta+2} v^{\beta+1}
|\X u|^2 X_l uD_jD_if(\X u)X_iX_luX_iv\, dx\\
& +\int_\Omega\sum_{i=1}^{2n} D_{l-n}D_if(\X u) TuX_i\big(\eta^{\beta+2}
v^{\beta+2}|\X u|^2 X_lu\big)\, dx\\
& -\int_\Omega T\big(D_{l-n}f(\X u)\big)\eta^{\beta+2}v^{\beta+2}|\X u|^2 X_lu\, dx\\
= &\,  I_1^l+I_2^l+I_3^l+I_4^l.
\end{aligned}
\end{equation}
Again we denote the integrals in the right hand side of (\ref{eq2}) by 
$I_1^l, I_2^l, I_3^l$ and $I_4^l$ in order respectively. 
Summing up the above equation (\ref{eq1}) and (\ref{eq2}) for all $l$ from
$1$ to $2n$, we end up with
\begin{equation}\label{equ3}
\int_\Omega \sum_{i,j,l} \eta^{\beta+2}v^{\beta+2}D_jD_if(\X u)X_jX_iuX_i\big( |\X u|^2 X_lu\big)\, dx=  \sum_l \sum_{m=1}^4 I_m^l.
\end{equation}
Here all sums for $i,j,l$ are from $1$ to $2n$. 

In the following, we estimate both sides of (\ref{equ3}). For the left hand of (\ref{equ3}), note that
\[ X_i\big(|\X u|^2 X_lu\big)=|\X u|^2 X_iX_l u+X_i(|\X u|^2)X_lu.\]
Then by the structure condition (\ref{structure}), we have that
\[ \sum_{i,j,l}D_jD_if(\X u)X_jX_luX_i\big(|\X u|^2X_lu\big)\ge 
\weight |\X u|^2|\X\X u|^2,\]
which gives us the following estimate for the left hand side of (\ref{equ3})
\begin{equation}\label{est:left}
\text{left of } (\ref{equ3}) \ge \int_\Omega \eta^{\beta+2}v^{\beta+2}
\weight |\X u|^2|\X\X u|^2\, dx.
\end{equation}

Then we estimate the right hand side of (\ref{equ3}). We will show that
$I_m^l$ satisfies the following estimate for each $l=1,2,...,2n$ and each $m=1,2,3,4$
\begin{equation}\label{est1}
\begin{aligned}
| I_m^l|\le & \, \frac{1}{36n}\int_\Omega \eta^{\beta+2}v^{\beta+2}\weight |\X u|^2|\X\X u|^2\,dx\\
 & + c (\beta+2)^2\int_\Omega \eta^\beta \big(|\X \eta|^2+\eta| T\eta|\big) v^{\beta+2}
\weight |\X u|^4 \, dx\\
& +  c(\beta+2)^2\int_\Omega \eta^{\beta+2} v^ \beta\weight |\X u|^4| \X v|^2\, dx\\
& +
c\int_\Omega\eta^{\beta+2} v^{\beta+2}\weight |\X u|^2| Tu|^2\, dx,
\end{aligned}
\end{equation}
where $c=c(n,p,L)>0$. Then the lemma follows from the above estimates 
(\ref{est:left}) and (\ref{est1})
for both sides of (\ref{equ3}). The proof of the lemma is finished, modulo
the proof of (\ref{est1}). In the rest, we prove (\ref{est1}) in the order of $m=1, 2,3,4$.

First, when $m=1$, we have for $I_1^l, l=1,2,...,2n$, by the structure condition (\ref{structure}) that
\[
| I_1^l| \le c(\beta+2)\int_\Omega \eta^{\beta+1}| \X \eta| v^{\beta+2}\weight |\X u|^3|\X\X u|\, dx,
\]
from which it follows by Young's inequality that
\begin{equation}\label{est:I1}
\begin{aligned}
| I_1^l|\le &\, \frac{1}{36n}\int_\Omega \eta^{\beta+2}v^{\beta+2}
\weight |\X u|^2|\X\X u|^2\,dx\\
 & + c(\beta+2)^2\int_\Omega \eta^\beta |\X \eta|^2 v^{\beta+2}
\weight |\X u|^4\, dx.
\end{aligned}
\end{equation}
Thus (\ref{est1}) holds for $I_1^l, l=1,2,...,2n$. 

Second, when $m=2$,  we have for $I_1^l, l=1,2,...,2n$, by the structure condition (\ref{structure}) that
\[ | I_2^l|\le c(\beta+2)\int_\Omega\eta^{\beta+2}v^{\beta+1}
\weight |\X u|^3|\X\X u\|\X v|\, dx,
\]
from which it follows by Young's inequality that
\begin{equation}\label{est:I2}
\begin{aligned}
| I_2^l|\le &\, \frac{1}{36n}\int_\Omega \eta^{\beta+2}v^{\beta+2}
\weight |\X u|^2|\X\X u|^2\,dx\\
 & + c(\beta+2)^2\int_\Omega \eta^{\beta+2} v^{\beta}
\weight |\X u|^4|\X v|^2\, dx.
\end{aligned}
\end{equation}
This proves (\ref{est1}) for $I_2^l, l=1,2,...,2n$.

Third, when $m=3$, we note that
\begin{equation*}
\begin{aligned}
\big| X_i & \big(  \eta^{\beta+2}v^{\beta+2}|\X u|^2 X_lu\big)\big|\le  3 \eta^{\beta+2} v^{\beta+2} |\X u|^2 |\X\X u|\\
& + (\beta+2)\eta^{\beta+1}v^{\beta+2}
|\X u|^3|\X \eta| +(\beta+2)\eta^{\beta+2}v^{\beta+1}
|\X u|^3| \X v|.
\end{aligned}
\end{equation*}
Thus by the structure condition (\ref{structure}), we have
\begin{equation*}
\begin{aligned}
| I_3^l| \le &\,  c \int_\Omega \eta^{\beta+2}v^{\beta+2}\weight |\X u|^2|\X \X u| | Tu|\, dx\\
& + c(\beta+2)\int_\Omega\eta^{\beta+1}|\X \eta| v^{\beta+2}\weight |\X u|^3| Tu|\, dx\\
& + c(\beta+2)\int_\Omega \eta^{\beta+2} v^{\beta+1}\weight |\X u|^3|\X v\| Tu|\, dx,
\end{aligned}
\end{equation*}
from which it follows by Young's inequality that
\begin{equation}\label{est:I3}
\begin{aligned}
| I_3^l| \le &\, \frac{1}{36n}\int_\Omega \eta^{\beta+2}v^{\beta+2}
\weight |\X u|^2|\X\X u|^2\,dx\\
&+c \int_\Omega \eta^{\beta+2}v^{\beta+2}\weight |\X u|^2 | Tu|^2\, dx\\
& + c(\beta+2)^2 \int_\Omega\eta^{\beta}|\X \eta|^2 v^{\beta+2}\weight |\X u|^4\, dx\\
& + c(\beta+2)^2\int_\Omega \eta^{\beta+2} v^{\beta}\weight |\X u|^4|\X v|^2\, dx.
\end{aligned}
\end{equation}
This proves (\ref{est1}) for $I_3^l, l=1,2,...,2n$.

Finally, when $m=4$, we prove (\ref{est1}) for $I_4^l$. We consider
only the case $l=1,2,...,n$. The case $l=n+1,n+2,...,2n$ can be treated similarly. 
Let 
\begin{equation}\label{def:w} 
w=\eta^{\beta+2}|\X u|^2 X_lu.
\end{equation}
Then we can write test-function $\varphi$ defined as in (\ref{varphi}) as
$\varphi=v^{\beta+2}w$. We rewrite $T$ as $T=X_1X_{n+1}-X_{n+1}X_1$. Then
integration by parts yields
\begin{equation}\label{est:I41}
\begin{aligned}
I_4^l=&\int_\Omega T\big(D_{n+l}f(\X u)\big)\varphi\, dx\\
=&\int_\Omega X_1\big(D_{n+l}f(\X u)\big) X_{n+1}\varphi-X_{n+1}\big(D_{n+l}f(\X u)\big)X_1\varphi\, dx.
\end{aligned}
\end{equation}
Note that 
$$\X \varphi= (\beta+2)v^{\beta+1}w\X v+ v^{\beta+2}\X w.$$ 
Thus (\ref{est:I41}) becomes
\begin{equation}\label{est:I42}
\begin{aligned}
I_4^l=&\, (\beta+2)\int_\Omega v^{\beta+1} w\Big(X_1\big(D_{n+l}f(\X u)\big)X_{n+1} v-X_{n+1}\big( D_{n+l}f(\X u)\big)X_1 v\Big)\, dx\\
&+\int_\Omega v^{\beta+2}\Big( X_1\big(D_{n+l}f(\X u)\big)X_{n+1}w-X_{n+1}\big(D_{n+l}f(\X u)\big)X_1w\Big)\, dx\\
= &\, J^l+K^l.
\end{aligned}
\end{equation}
Here we denote the first and the second integral in the right hand side of 
(\ref{est:I41}) by $J^l$ and $K^l$, respectively. We estimate 
$J^l$ as follows. By the structure condition (\ref{structure}) and the
definition of $w$ as in (\ref{def:w}),
\begin{equation*}
| J^l| \le c(\beta+2)\int_\Omega \eta^{\beta+2} v^{\beta+1}
\weight |\X u|^3| \X\X u\|\X v|\, dx,
\end{equation*}
from which it follows by Young's inequality, that
\begin{equation}\label{est:Jl}
\begin{aligned}
| J^l| \le &\, \frac{1}{72n}\int_\Omega \eta^{\beta+2}v^{\beta+2}\weight |\X u|^2|\X\X u|^2\,dx\\
&+c(\beta+2)^2\int_\Omega \eta^{\beta+2} v^{\beta}
\weight |\X u|^4|\X v|^2\, dx.
\end{aligned}
\end{equation}
The above inequality shows that $J^l$ satisfies similar estimate as (\ref{est1}) for all $l=1,2,...,n$.
Then we estimate $K^l$. Integration by parts again, yields 
\begin{equation}\label{est:Kl}
\begin{aligned}
K^l=&\, (\beta+2) \int_\Omega v^{\beta+1} D_{n+l}f(\X u)\Big( X_{n+1} vX_1 w-X_1vX_{n+1}w\Big)\, dx\\
&-\int_\Omega v^{\beta+2}D_{n+l}f(\X u) Tw\, dx\\
= &\, K_1^l+K_2^l.
\end{aligned}
\end{equation}
For $K_1^l$, we have by the structure condition (\ref{structure}) that
\begin{equation*}
\begin{aligned}
| K^l_1| \le\ &  c(\beta+2)\int_\Omega \eta^{\beta+2}v^{\beta+1}\weight |\X u|^3|\X\X u\|\X v|\, dx\\
& +c(\beta+2)^2\int_\Omega\eta^{\beta+1}v^{\beta+1}\weight |\X u|^4|\X v\|\X \eta|\, dx
\end{aligned}
\end{equation*}
from which it follows by Young's inequality that
\begin{equation}\label{est:Kl1}
\begin{aligned}
| K^l_1| \le &\, \frac{1}{144n}\int_\Omega \eta^{\beta+2}v^{\beta+2}\weight |\X u|^2|\X\X u|^2\,dx\\
&+c(\beta+2)^2\int_\Omega \eta^{\beta+2} v^{\beta}
\weight |\X u|^4|\X v|^2\, dx\\
& + c(\beta+2)^2\int_\Omega \eta^{\beta}|\X \eta|^2 v^{\beta+2}\weight |\X u|^4\, dx.
\end{aligned}
\end{equation}
The above inequality shows that $K_1^l$ also satisfies similar estimate as (\ref{est1}) for all $l=1,2,...,n$. We continue to estimate $K_2^l$ in (\ref{est:Kl}). Note that
\begin{equation*}
\begin{aligned}
Tw=\, (\beta+2)\eta^{\beta+1}|\X u|^2 X_luT\eta+\eta^{\beta+2}|\X u|^2 X_lTu + \sum_{i=1}^{2n}2\eta^{\beta+2}X_luX_iuX_iTu.
\end{aligned}
\end{equation*}
Therefore we write $K^l_2$ as
\begin{equation*}
\begin{aligned}
K_2^l = &\, -(\beta+2)\int_\Omega\eta^{\beta+1}v^{\beta+2} D_{n+l} f(\X u)|\X u|^2 X_luT\eta\, dx\\
&\, -\int_\Omega \eta^{\beta+2}v^{\beta+2} D_{n+l}f(\X u)|\X u|^2 X_lTu\, dx\\
&-2\int_\Omega\eta^{\beta+2}v^{\beta+2} D_{n+l} f(\X u)X_luX_iu X_iTu\, dx.
\end{aligned}
\end{equation*}
For the last two integrals in the above equality, we apply integration by parts.
We obtain that
\begin{equation*}
\begin{aligned}
K_2^l = &\, -(\beta+2)\int_\Omega\eta^{\beta+1}v^{\beta+2} D_{n+l} f(\X u)|\X u|^2 X_luT\eta\, dx\\
&\, +\int_\Omega X_l\Big(\eta^{\beta+2}v^{\beta+2} D_{n+l}f(\X u)
\vert\X u\vert^2\Big)Tu\, dx\\
&+2\int_\Omega X_i\Big(\eta^{\beta+2}v^{\beta+2} D_{n+l} f(\X u)X_luX_iu\Big)Tu\, dx.
\end{aligned}
\end{equation*}
Now we may estimate the integrals in the above equality by the structure condition (\ref{structure}). We obtain the following estimate for $K^l_2$.
\begin{equation*}
\begin{aligned}
| K_2^l|\le &\, c(\beta+2)\int_\Omega \eta^{\beta+1}v^{\beta+2}\weight |\X u|^4| T\eta|\, dx\\
&\, +c\int_\Omega \eta^{\beta+2}v^{\beta+2}\weight |\X u|^2|\X\X u\| Tu|\, dx\\
&\, +c(\beta+2)\int_\Omega \eta^{\beta+2} v^{\beta+1}\weight |\X u|^3|\X v\| Tu|\, dx\\
&\, +c(\beta+2)\int_\Omega \eta^{\beta+1}v^{\beta+2}\weight |\X u|^3|\X\eta\| Tu|\, dx.
\end{aligned}
\end{equation*}
By Young's inequality, we end up with the following estimate for $K_2^l$
\begin{equation}\label{est:K2l}
\begin{aligned}
| K_2^l|\le & \, \frac{1}{144n}\int_\Omega \eta^{\beta+2}v^{\beta+2}\weight |\X u|^2|\X\X u|^2\,dx\\
 & + c (\beta+2)^2\int_\Omega \eta^\beta \big(|\X \eta|^2+\eta| T\eta|\big)v^{\beta+2}
\weight |\X u|^4\, dx\\
& +  c(\beta+2)^2\int_\Omega \eta^{\beta+2} v^ \beta\weight |\X u|^4| \X v|^2\, dx\\
& +
c\int_\Omega\eta^{\beta+2} v^{\beta+2}\weight |\X u|^2| Tu|^2\, dx.
\end{aligned}
\end{equation}
This shows that $K_2^l$ also satisfies similar estimate as (\ref{est1}). Now we combine the estimates (\ref{est:Kl1}) for $K^l_1$ and (\ref{est:K2l}) for $K_2^l$. Recall that $K^l$ is the sum of $K_1^l$ and $K_2^l$ as denoted in (\ref{est:Kl}). We obtain that the following estimate for $K^l$.
\begin{equation}\label{est:Kll}
\begin{aligned}
| K^l|\le & \, \frac{1}{72n}\int_\Omega \eta^{\beta+2}v^{\beta+2}\weight |\X u|^2|\X\X u|^2\,dx\\
 & + c (\beta+2)^2\int_\Omega \eta^\beta \big(|\X \eta|^2+\eta| T\eta|\big)v^{\beta+2}
\weight |\X u|^4\, dx\\
& +  c(\beta+2)^2\int_\Omega \eta^{\beta+2} v^ \beta\weight |\X u|^4| \X v|^2\, dx\\
& +
c\int_\Omega\eta^{\beta+2} v^{\beta+2}\weight |\X u|^2| Tu|^2\, dx.
\end{aligned}
\end{equation}
Recall that $I_4^l$ is the sum of $J^l$ and $K^l$. We combine the estimates (\ref{est:Jl}) for $J^l$ and (\ref{est:Kll}) for $K^l$, and we can see that the claimed estimate (\ref{est1}) holds for $I_4^l$ for all $l=1,2,...,n$. We can prove 
(\ref{est1}) similarly for $I_4^l$ for all $l=n+1,n+2,...,2n$. This finishes the proof of the claim (\ref{est1}) for $I_m^l$ for all $l=1,2,...,2n$ and all $m=1,2,3,4$, and hence also the proof of the lemma.
\end{proof}

\end{document}